\documentclass[reqno,12pt]{amsart}
\DeclareFontFamily{OML}{script}{}{}
\DeclareFontShape{OML}{script}{m}{it}
{ <5-20> rsfs10 }{}
\DeclareMathAlphabet{\mathscript}{OML}{script}{m}{it}
\usepackage{cite}
\usepackage{color}

\usepackage[pagewise]{lineno}

\usepackage{graphicx,graphics}
\usepackage[colorlinks=true]{hyperref}
\usepackage{amstext}
\usepackage{pifont}
\usepackage{amsthm}
\usepackage{bm}
\usepackage{amscd}
\usepackage{amsmath}
\usepackage{amssymb}
\usepackage{latexsym}
\usepackage{amsfonts}
\usepackage[notref,final]{showkeys}
\usepackage{esint}
\usepackage{enumerate}

\usepackage[ansinew]{inputenc}
\usepackage[encapsulated]{CJK}

\ifx\text\undefined
\newcommand{\text}{\mbox}
\fi
\ifx\operatorname\undefined
\newcommand{\operatorname}{\mathop}
\fi

\newcommand{\ki}{{\mbox{\raise.5ex\hbox{$\chi$}}\hspace{.2ex}}}

\renewcommand{\epsilon}{\varepsilon}

\renewcommand{\hat}{\widehat}
\newenvironment{eqa}{\begin{equation}%
		\begin{array}{rcl}}{\end{array}\end{equation}}
\newcommand\beqa{\begin{eqa}}
	\newcommand\eeqa{\end{eqa}}

\numberwithin{equation}{section}
\newtheorem{thm}{Theorem}[section]
\newtheorem{alphatheorem}{Theorem}  
  
\newtheorem{cor}[thm]{Corollary}
\newtheorem{lem}{Lemma}[section]
\newtheorem{rmk}{Remark}[section]

\theoremstyle{definition}
\newtheorem{defn}[lem]{Definition}

\newcommand{\thmref}[1]{Theorem~\ref{#1}}
\newcommand{\secref}[1]{Section \ref{#1}}
\newcommand{\lemref}[1]{Lemma~\ref{#1}}

\newcommand{\corref}[1]{Corollary~\ref{#1}}

\newcommand{\void}[1]{}

\newcommand\BR{{\mathbb R}}
\newcommand\BN{{\mathbb N}}

\newcommand{\bR}{{\mathbb R}}  

\begin{document}\begin{CJK}{UTF8}{gkai}

		\title[  Gradient Estimates for  $	\Delta_{p}(u^{\gamma})+ au^{q}=0$ ]{Gradient Estimates for the doubly nonlinear diffusion equation on Complete Riemannian Manifolds }
		\author{Chen Guo and Zhengce Zhang}
		\date{\today}
		\address[Chen Guo]{School of Mathematics and Statistics, Xi'an Jiaotong University, Xi'an, 710049, P. R. China}
		\email{jasonchen123@stu.xjtu.edu.cn}
	
		\address[Zhengce Zhang]{School of Mathematics and Statistics, Xi'an Jiaotong University, Xi'an, 710049, P. R. China}
		\email{zhangzc@mail.xjtu.edu.cn}
		\thanks{Corresponding author: Zhengce Zhang}
		\thanks{Keywords: Doubly nonlinear diffusion equation; Liouville-type theorem;  Gradient estimate}
		\thanks{2020 Mathematics Subject Classification: 35B09, 35J92, 35R01, 53C21}

		\begin{abstract}
					We study the elliptic version of doubly nonlinear  diffusion equations on a complete Riemannian manifold $(M,g)$. Through the combination of a special nonlinear transformation and the standard Nash-Moser iteration procedure, some Cheng-Yau type gradient estimates for positive solutions are derived. As by-products, we also obtain Liouville type results and Harnack's inequality. These results fill a gap in Yan and Wang (2018)\cite{YW}, due to the lack of one key inequality when $b=\gamma-\frac{1}{p-1}>0$,  and provide a partial answer to the question that whether gradient estimates for the doubly nonlinear diffusion equation can be extended to the case $b>0$ . 
		\end{abstract}
		\maketitle		
		\section{Introduction}\label{Introduction}
		The investigation of gradient estimates for elliptic and parabolic equations on Riemannian manifolds has a rich historical background. Li and Yau \cite{LY1} first established well-known gradient estimates for the linear heat equation on Riemannian manifolds. This seminal paper, together with the corresponding elliptic result \cite{CY}, has a profound and long-lasting impact on subsequent research and a wide range of applications. For example, the key idea can be used to estimate  eigenvalues of a manifold \cite{LY2}, the lower and upper bound of the heat kernel \cite{CLS} and investigate the geometric properties of a manifold \cite{KN}.
 			
		Their approach has been further developed. On the one hand, people obtained  different types of estimates such as Davies type\cite{DEB}, Hamilton type\cite{HRS}, Souplet-Zhang type \cite{SZ1} and Li-Xu type \cite{LX} for various nonlinear equations on Riemannian manifolds. These nonlinear equations, arising from geometry, physics, non-Newtonian fluids and various other categories, have been deeply studied by many scholars.	
		On the other hand, regardless of whether these equations are defined on  Riemannian manifolds or more general geometry structures such as sub-Riemannian manifolds \cite{WZ2}, graphs\cite{BHLLMY}, Alexandrov spaces\cite{ZZ}, and metric measure spaces\cite{AA,ZY1}, researchers also obtained the corresponding gradient estimates. It is important to note that the Li-Yau gradient estimate is a fundamental element in the derivation of the entropy formula for Ricci flow.
					
		 	In this paper, we employ the Nash-Moser iteration technique to derive some Cheng-Yau type gradient estimates for the following doubly nonlinear diffusion equation 
			\begin{equation}\label{main equ}
				\Delta_{p}(u^{\gamma})+ au^{q}=0
			\end{equation}
		    on a complete Riemannian manifold $(M^{n},g)$, where $p>1$, $n\ge 2$, $\gamma>0$, $a,q\in \bR$ and
		    \begin{equation*}\label{def of double diffusion}
		    	\Delta_{p}(u^{\gamma})\triangleq div(|\nabla(u^{\gamma})|^{p-2}\nabla(u^{\gamma})).
		    \end{equation*}
		  		
			Our motivation primarily arises from two key aspects. One is \eqref{main equ} has abundant backgrounds. Its counterpart parabolic equation   
			for $a=0$ is
			\begin{equation}\label{parabolic version of main equ}
				\frac{\partial u}{\partial t}=\Delta_{p}(u^{\gamma}).
			\end{equation}
			
			As a generalization of the heat equation ($p=2,\gamma=1$), $p$-Laplace heat equation ($\gamma=1$), the porous medium equation ($p=2,\gamma>1$)
			and fast diffusion equation ($p=2,\gamma<1$), \eqref{parabolic version of main equ} appears in natural phenomena like non-Newtonian fluids, turbulent flows in porous media and glaciology\cite{VJ}. It is also closely related to various types of geometric flows such as Yamabe flow\cite{CLN} and inverse curvature flow\cite{MR1}.
			
			It is worth mentioning that Aronson and B\'{e}nlian \cite{AB} derived the following second order differential inequality for positive smooth solutions of the porous medium equation in Euclidean space $\BR^{n}$ 
			\begin{equation}\label{A-B estimate}
			     \sum_{i}\frac{\partial}{\partial x_{i}}\left(\gamma u^{\gamma-2}\frac{\partial u}{\partial x_{i}}\right)\ge -\frac{\kappa}{t},
			\end{equation}
			where $\gamma>1-\frac{2}{n}$ and $\kappa=\frac{n}{n(\gamma-1)+2}$.	
			Later, Lu, Ni, V\'{a}zquez and Villani \cite{LPNVV} extended it to both the porous medium equation and fast diffusion equation on Riemannian manifolds. They also got some Li-Yau type gradient estimates and entropy formulae.
			
			Corresponding results are given for the doubly nonlinear diffusion equation as well. Wang and Chen \cite{WC} got a sharp Li-Yau type gradient estimate and an entropy monotonicity
			formula on compact Riemannian manifolds with nonnegative Ricci curvature.
			Chen and Xiong \cite{CX} proved Li-Xu type, Davies type, and Hamilton type gradient estimates for \eqref{parabolic version of main equ}. Elliptic type gradient estimates on both compact and complete noncompact Riemannian manifolds were established by Yan and Wang\cite{YW}.
			\begin{alphatheorem}{\rm(\hspace{-0.03em}\cite[Theorem 4.1]{YW})}\label{thmA}
				Let $(M^{n},g)$ be an $n$-dimensional complete noncompact Riemannian manifold with the sectional curvature bounded from below by $-K^{2}$ for some  nonnegative constant $K$. Provided $u$ is a positive solution to \eqref{parabolic version of main equ} with $1<p\le 2$ and the upper bound $ u\le \exp(-\frac{1}{\xi \gamma\,e})$
				for some positive constant $\xi$.  $b=\gamma-\frac{1}{p-1}$ satisfies $-\tfrac{p-1}{a}<b<0$, where 
				\begin{equation*}
					a=\max\left\{\frac{(n+4)(p-1)^{2}}{4},\frac{n(p-1)}{2}+\frac{n\xi}{4}\right\}.
				\end{equation*}
				Then for all $(x,t)\in M\times (0,\infty)$, 
				\begin{equation}\label{Wang's elliptic estimate}
					\frac{|\nabla v|^{p}}{(1-v)^{p}}(x,t)\le C_{1}K^{p}\Theta^{\frac{p}{2}}+\frac{C_{2}}{t},
				\end{equation}
				where $v=\frac{\gamma}{b}u^{b}$,
				\[\Theta=\sup_{(x,t)\in M\times (0,\infty)}(bv)<\infty , \]
				and  $C_{1},C_{2}$ are constants depending on $n,p,a,\gamma$.
			\end{alphatheorem}

			Nash-Moser iteration, as a milestone in the history of PDE, provides us with an elegant way to study linear elliptic and parabolic equations with only measurable coefficients, which links the regularity problem in the nonlinear case. Wang and Zhang \cite{WZ} utilized this technique and gained local gradient estimates for $p$-harmonic functions on complete Riemannian manifolds
			\begin{equation}\label{CY estimate}
				\sup_{B(o,\frac{R}{2})} \frac{|\nabla u|}{u}\leq C(n,p)\frac{(1+\sqrt K R)}{R}.
			\end{equation}
			
				The above conclusion generalizes Cheng-Yau's classical result \cite{CY} for harmonic functions.  Distinct from \cite{KN}, their method only assumes the lower bound of Ricci curvature. The optimal constant $C(n,p)=(n-1)/(p-1)$ is given by Sung and Chang\cite{SW}. In the case $\gamma=1$, \eqref{main equ} reduces to the famous Lane-Emden equation which is fully considered by plenty of researchers.
				Representative works in $\BR^{n}$ include Gidas and Spruck \cite{GS} for $p=2$, and Serrin and Zou \cite{SZ} for the general $p>1$. For results on Riemannian manifolds, we refer to \cite{HWW,HGG,MHL,WW2} and the references therein. 
				 He, Hu and Wang\cite{he2023nashmoser} studied the following equation involving the power of $u$ and its gradient		
			\begin{equation}\label{eilliptic 1}
				\Delta_{p} u+ \beta u^{r} |\nabla u|^{q}=0.
			\end{equation} 
			
			They got the following result by iterating the norm of $|\nabla u|^{2}$ on a series of geodesic balls and using the integral estimate for some suitable norm of $u$.
			\begin{alphatheorem}{\rm(\hspace{-0.03em}\cite[Theorem 1.2]{he2023nashmoser})}\label{thmB}
				Let $(M^{n},g)$ be an $n$-dimensional complete Riemannian manifold with Ricci curvature bounded from below by $-(n-1)K$, where $K\ge 0$ is a constant. Assume $p>1$ and $u$ is a $C^{1}$- positive weak solution of \eqref{eilliptic 1} on a geodesic ball $B(o,R)$. If $n,p,q,r$ and $\beta$ satisfy 
				\begin{gather*}
					\beta\left(\frac{n+1}{n-1}-\frac{q+r}{p-1}\right)\ge 0,
				\intertext{or}
				1<\frac{q+r}{p-1}<\frac{n+3}{n-1}, \;\forall \beta\in R,
				\end{gather*}
				then there exists some positive constant $C(n,p,q,r)$ relying on $n,p,q$ and $r$ such that 
					\begin{equation}\label{CY estimate for p-laplace}
					\sup_{B(o,\frac{R}{2})} \frac{|\nabla u|}{u}\leq C(n,p,q,r)\frac{(1+\sqrt K R)}{R}.
				\end{equation}
			\end{alphatheorem}
			Their approach eliminates cumbersome restrictions on the range of parameters $\beta,p,q,r$ and avoids complicated computation by the classical Bernstein method. This conclusion is novel even in Euclidean space compared with the previous results\cite{BMGL,FPS,F}. Now a natural question arises whether this method could be applied to equation \eqref{main equ}. 
			
			Another motivation originates from two questions mentioned in \cite{YW}. Yan and Wang asked whether the gradient estimate could be obtained only under the assumption on Ricci curvature and be extended to the case $b>0$, i.e. $\gamma>1/(p-1)$. Actually,   the obstacle appears in the following inequality
			\begin{equation}\label{obst}
				-\frac{n+4}{4}(p-1)^{2}bv+\frac{n(p-1)}{2}b-\frac{nb}{4v}\ge ab(1-v).
			\end{equation}
			
			 Notice that in \thmref{thmA}, $a>\frac{(n+4)(p-1)^{2}}{4}$ and the bound of $u$  implies $v<-\frac{1}{\xi}$. To ensure \eqref{obst} holds true, using the truth 
			 \[-\frac{nb}{4v}\ge \frac{n\xi b}{4},\]
			 one needs to guarantee that
			 \[ \frac{n(p-1)}{2}b+\frac{n\xi b}{4}\ge ab.\]
			 
			 Henceforth, the definition of $a$ suggests that $b$ has to be negative. In addition,  the constraint $1<p\le 2$  arises from the technical requirement of deriving the inequality of $\varphi G$ in the proof of  \thmref{thmA}.
			 
			Thanks to the Nash-Moser iteration technique, these inquiries can be partially addressed through the following conclusions. 
			Now we state the main results contained in this paper. The first  theorem elucidates the case $\gamma> 1/(p-1)$.
			\begin{thm}\label{thm1.1}
				Let $(M^{n},g)$ be a complete Riemannian manifold satisfying $Ric(M)\ge -(n-1)K g$ for some constant $K\ge 0$. Assume $u$ is a $C^{1}$-positive solution to  equation \eqref{main equ} on the geodesic ball $B(o,R)\subset M$. Denote $b=\gamma-\tfrac{1}{p-1}> 0$. If  $a,n,p,q$ and $\gamma$ satisfy one of the following conditions
				\begin{equation}\label{condition1 in thm1.1}
					a\left[\frac{n+1}{n-1}+\frac{2-(n-1)(q-1)}{b(n-1)(p-1)}\right]\ge 0\,;
				\end{equation}
				\begin{equation}\label{condition2 in thm1.1}
					\gamma(p-1)<q<\frac{n+3}{n-1}\gamma(p-1), \quad \forall a\in \BR,
				\end{equation}

				then there exists a constant $C=C(n,p,q,\gamma)>0$ such that
				$$
				\sup_{B(o,\frac{R}{2})} \frac{|\nabla u|}{u}\leq C\frac{(1+\sqrt K R)}{R}.
				$$
			\end{thm}			 
			
			\noindent The following theorem concentrates on the case $0<\gamma<\frac{1}{p-1}$.	
			\begin{thm}\label{thm1.2}
			Let $(M^{n},g)$ be a complete Riemannian manifold satisfying $Ric(M)\ge -(n-1)K g$ for some constant $K\ge 0$. Assume $u$ is a $C^{1}$-positive solution to  equation \eqref{main equ} on the geodesic ball $B(o,R)\subset M$. Denote $b=\gamma-\tfrac{1}{p-1}< 0$. If  $a,n,p,q$ and $\gamma$ satisfy one of the following conditions
	\begin{equation}\label{condition1 in thm1.2}
		a\left[\frac{n+1}{n-1}+\frac{2-(n-1)(q-1)}{b(n-1)(p-1)}\right]\le 0\,;
	\end{equation}
	\begin{equation}\label{condition2 in thm1.2}
		\gamma(p-1)<q<\frac{n+3}{n-1}\gamma(p-1), \quad \forall a\in \BR,
	\end{equation}
	then there exists a constant $C=C(n,p,q,\gamma)>0$ such that
	$$
	\sup_{B(o,\frac{R}{2})} \frac{|\nabla u|}{u}\leq C\frac{(1+\sqrt K R)}{R}.
	$$
\end{thm}			 

			As a supplement of \thmref{thm1.1} and \thmref{thm1.2}, the following conclusion concerns the remaining case $\gamma=\frac{1}{p-1}$.
			 \begin{thm}\label{thm1.3}
			 			Let $(M^{n},g)$ be a complete Riemannian manifold satisfying $Ric(M)\ge -(n-1)K g$ for some constant $K\ge 0$. Assume $u$ is a $C^{1}$-positive solution to  equation \eqref{main equ} on the geodesic ball $B(o,R)\subset M$. Denote $\gamma=\frac{1}{p-1}$, i.e. $b=0$. If  $a,n,p$ and $q$ satisfy one of the following conditions
			 		\begin{equation}\label{condition1 in thm1.3}
			 			\frac{2a(p-1)}{n-1}-a(p-1)(q-1)\ge 0\,;
			 		\end{equation}
			 		\begin{equation}\label{condition2 in thm1.3}
			 			1<q<\frac{n+3}{n-1},\quad \forall a\in \BR,
			 		\end{equation}
			 		
			 		then there exists a constant $C=C(n,p,q,\gamma)>0$ such that
			 		$$
			 		\sup_{B(o,\frac{R}{2})} \frac{|\nabla u|}{u}\leq C\frac{(1+\sqrt K R)}{R}.
			 		$$
			 \end{thm}
			Through a careful analysis of the conditions in the above theorems, we have the following corollary.
			\begin{cor}\label{cor1.4}
					Let $(M^{n},g)$ be a complete Riemannian manifold satisfying $Ric(M)\ge -(n-1)K g$ for some constant $K\ge 0$. Assume $u$ is a $C^{1}$-positive solution to  equation \eqref{main equ} on the geodesic ball $B(o,R)\subset M$.
					If
					\begin{align*}
						&a>0 \quad \text{and} \quad q<\frac{n+3}{n-1}\gamma(p-1),
						\intertext{or}
						&a<0 \quad \text{and} \quad q>\gamma(p-1),
					\end{align*}

					then
						$$
					\sup_{B(o,\frac{R}{2})} \frac{|\nabla u|}{u}\leq C\frac{(1+\sqrt K R)}{R}.
					$$
			\end{cor}
			 \begin{rmk}
			 When $\gamma=1$, our results reduce to the conclusions in He-Wang-Wei{\rm\cite{HWW}}. More precisely,  \thmref{thm1.3} reduces to  their results in the boardline case $p=2$. For $p\neq 2$, notice that the following formula holds 
			 \begin{align*}
			 	\frac{n+1}{n-1}+\frac{2-(n-1)(q-1)}{b(n-1)(p-1)}&=	\frac{n+1}{n-1}+\frac{2-(n-1)(q-1)}{(p-2)(n-1)}\\
			 	&=\frac{n+1}{n-1}-\frac{q}{p-1}+\frac{2}{(p-2)(n-1)}+\frac{q}{p-1}-\frac{q-1}{p-2}\\
			 	&=\left(1+\frac{1}{p-2}\right)\left(\frac{n+1}{n-1}-\frac{q}{p-1}\right).
			 \end{align*}
			 Applying this observation, we note that \thmref{thm1.1} recovers their findings in the range $p>2$, whereas \thmref{thm1.2} coincides with their results when  $1<p<2$.
			 \end{rmk}  
			 As by-products of three theorems above, we can directly obtain  some Liouville type results and  Harnack's inequalities.
			 \begin{cor}\label{liouville}
			 	Let $(M^{n},g)$ be a complete noncompact Riemannian manifold with nonnegative Ricci curvature. If $a,n,p,q$ and $\gamma$ satisfy  the conditions of one of the above theorems, then equation \eqref{main equ} admits no positive solutions.
			 \end{cor}
			 \begin{cor}\label{harnack}
			 Same notations and assumptions in one of the above theorems, assume $u$ is a positive solution to  equation \eqref{main equ} on the geodesic ball $B(o,R)\subset M$ with $a,n,p,q$ and $\gamma$ satisfying conditions in corresponding theorem, then for any $y,z\in B(o,R/2)$ one has
			 \begin{equation*}
			    \log\frac{u(z)}{u(y)}\le C(n,p,q,\gamma)(1+\sqrt{K}R).
			 \end{equation*}
			 Moreover, if $u$ is defined on $M$, there holds
			 \begin{equation*}
			 	  \frac{u(z)}{u(y)}\le e^{C(n,p,q,\gamma)\sqrt{K}d(y,z)},
			 \end{equation*}
			  where $d(y,z)$ is the geodesic distance between $y$ and $z$.
			 \end{cor}
			 
			 In the last section of this article, we discuss the special case $a=0$. Instead of $u$, we take the local gradient estimate of $v$ into consideration. Some new conclusions are obtained under the case $\gamma>\frac{1}{p-1}, \gamma=\frac{1}{p-1}, \;\text{and}\; \gamma<\frac{1}{p-1}$ respectively by an ingenious application of \thmref{thmB} (see \thmref{thm6.1} for the first two cases and \thmref{thm6.4} for the last one respectively).  One can directly trace back to $u$ in the first case $\gamma>\frac{1}{p-1}$, while the remaining two  require extra conditions on $n,p$ and solution $u$ itself. In addition, a Caccioppoli type inequality is established in the case $\gamma<\frac{1}{p-1}$ and some Liouville type results are also obtained. One can see \thmref{thm6.2} and \corref{cor6.3} for the details.
			 
			 This paper is organized as follows. In Section 2, we list some necessary lemmas and do some preparatory work. In Section 3, we will give a lower bound estimate for linearization operator $\mathcal{L}$  and give a further discussion about  conditions for parameters $a,p,\gamma$ and $q$. Some technical lemmas will be derived. The core of the proof, involving integral estimate and the iteration process, will be given in Section 4. In Section 5, we prove the aforementioned results. Furthermore, we will discuss the special case  $a=0$ in Section 6.

\section{Preliminaries}
\subsection{Notations}
Throughout this paper,  $(M^{n},g)$ is an $n$-dimensional Riemannian manifold. $\mathrm{d}v_{g}$ denotes its standard volume form. The integral of a function $u$ over $M$ is written as 
\begin{gather*}
	\int_{M} u \;\mathrm{d}{v}_{g}.
\end{gather*}

Hereinafter we will omit the volume form of integral over $M$ for simplicity.	In the below the letters $c_{1},c_{2},c_{3},\cdots$  denote some positive constants relying on $n,p,q$ and $\gamma$, which may change value from line to line. $K$ is some nonnegative constant and $C(\cdot)$ means some positive constant that depends on some parameters in the bracket.
\begin{defn}\label{def of solution}
	A $C^{1}$ solution $u$ is said to be a positive weak solution of equation \eqref{main equ} in a domain $D$ if for all $\phi\in C^{\infty}_{0}(D)$, there exists
	\begin{equation*}
	-\int_{D} |\nabla(u^{\gamma})|^{p-2}\langle \nabla(u^{\gamma}),\nabla \phi \rangle+\int_{D} au^{q}\phi=0.
	\end{equation*}
\end{defn}
		As in \cite{WC}, we rewrite \eqref{main equ}  as a diffusion equation with  the destiny $u\ge 0$
		\begin{equation*}
			\mathrm{div}(c(u,\nabla u)\nabla u)+au^{q}=0,
		\end{equation*}
		where $c(u,\nabla u)\triangleq \gamma^{p-1}|u|^{(p-2)(\gamma-1)}u^{\gamma-1}|\nabla u|^{p-2}$ is the diffusion coefficient. Because this equation is apparently degenerate when $u=0$ or $|\nabla u|=0$, we always carry out the computation in the domain where $u$ and $|\nabla u|$ retain positive. We refer the readers to the monograph\cite{VJ,MJ} for an account of the regularity of the doubly nonlinear diffusion equation. 
		
Next, we recall  Saloff-Coste's Sobolev inequalities (see \cite[Theorem 3.1]{SL1}). It plays a significant rule in the iteration process.
\begin{lem}\label{salof}
	Let $(M^{n},g)$ be an $n$-dimensional complete Riemannian  manifold satisfying $Ric\geq-(n-1)K g$, where $K$ is a nonnegative constant. For $n>2$, there exists a positive constant $C_n$ which only depends on $n$, such that for any geodesic ball $B\subset M$ of radius R and volume $V$,  
	$$
	\|f\|_{L^{\frac{2n}{n-2}}}^2\leq e^{C_n(1+\sqrt{K}R)}\,V^{-\frac{2}{n}}R^2\left(\int|\nabla f|^2+R^{-2}f^2\right)
	$$
	is valid for $f\in C^{\infty}_0(B)$. For $n\le 2$, the above inequality still holds with $n$ replaced by any fixed $\hat{n}>2$.
\end{lem}

	\subsection{Some transformations }\label{section2.2} We begin to transform  equation \eqref{main equ}. Firstly, we set $b=\gamma-\frac{1}{p-1}$ and use the following change of variable 
	\begin{align}\label{equ 2.1}
		v=\begin{cases}
			\frac{\gamma}{b}u^{b}, \qquad &b>0,\\
			\frac{1}{p-1}\log u, \qquad &b=0,\\
			-\frac{\gamma}{b}u^{b},\qquad &b<0,
		\end{cases}
	\end{align}
	where $v$ is called ``pressure'' in the physics literature (see \cite{WC,LPNVV} for more explanation). Then $v$ satisfies
	\begin{align}
		&\Delta_{p}v+b^{-1}v^{-1}|\nabla v|^{p}+a\left(\frac{b}{\gamma}\right)^{\frac{q-1}{b}}v^{\frac{q-1}{b}}=0, &b> 0;\label{equation1 for v} \\
		&\Delta_{p}v+(p-1)|\nabla v|^{p}+ae^{(p-1)(q-1)v}=0,& b=0,\label{equation2 for v} \\
		&\Delta_{p}v+b^{-1}v^{-1}|\nabla v|^{p}-a\left(-\frac{b}{\gamma}\right)^{\frac{q-1}{b}}v^{\frac{q-1}{b}}=0, &b< 0.\label{equation3 for v} 
	\end{align}
	
	In the case $b\neq 0$, we apply the logarithmic transformation $\omega=-(p-1)\log v$ to equations \eqref{equation1 for v} and \eqref{equation3 for v} respectively. Then equation \eqref{equation1 for v} becomes
	\begin{equation*}
		\Delta_{p}\omega-\left(1+\frac{1}{b(p-1)}\right)|\nabla \omega|^{p}-a(p-1)^{p-1}\left(\frac{b}{\gamma}\right)^{\frac{q-1}{b}}e^{\left(1-\frac{q-1}{b(p-1)}\right)\omega}=0,
	\end{equation*}
	and equation \eqref{equation3 for v} transforms into
	\begin{equation*}
		\Delta_{p}\omega-\left(1+\frac{1}{b(p-1)}\right)|\nabla \omega|^{p}+a(p-1)^{p-1}\left(-\frac{b}{\gamma}\right)^{\frac{q-1}{b}}e^{\left(1-\frac{q-1}{b(p-1)}\right)\omega}=0.
	\end{equation*}
	For convenience, we denote
	\[ c=1+\frac{1}{b(p-1)},\quad d=a(p-1)^{p-1}\left(\frac{b}{\gamma}\right)^{\frac{q-1}{b}},\quad l=a(p-1)^{p-1}\left(-\frac{b}{\gamma}\right)^{\frac{q-1}{b}}\]
		and \[ k=1-\frac{q-1}{b(p-1)}.\]
	Henceforth,  equations \eqref{equation1 for v} and \eqref{equation3 for v} can be rewritten as
	\begin{equation}\label{equ1 for omega}
        \Delta_{p}\omega-c|\nabla \omega|^{p}-de^{k\omega}=0,
	\end{equation}
	\begin{equation}\label{equ3 for omega}
		\Delta_{p}\omega-c|\nabla \omega|^{p}+le^{k\omega}=0,
	\end{equation}
	For $b=0$, let $\omega=v$ directly and \eqref{equation2 for v} becomes
	\begin{equation}\label{equ2 for omega} 
		\Delta_{p}\omega+(p-1)|\nabla \omega|^{p}+ae^{(p-1)(q-1)\omega}=0.
	\end{equation}	
Define the linearization operator $\mathcal{L}$ of $p$-Laplacian 
\begin{align}
	\mathcal{L}(\psi)= div \left(f^{p/2-1} A(\nabla \psi) \right ),
\end{align}

\noindent where $f = |\nabla \omega|^2$ and
\begin{align}\label{defofA}
	A(\nabla\psi) = \nabla\psi+(p-2)f^{-1}\langle \nabla \psi,\nabla \omega\rangle \nabla \omega.
\end{align}

The following lemma is closely related to the expression of $\mathcal L(f^\alpha)$ for any $\alpha>0$.	
See \cite[Lemma 2.3]{HWW} for its proof. 
\begin{lem}
	For any $\alpha >0$, the equality
	\begin{align}
		\label{lemma 2.2}
		\begin{split}
			\mathcal{L} (f^{\alpha}) =
			&
			\alpha\left(\alpha+\frac{p}{2}-2\right)f^{\alpha+\frac{p}{2}-3}|\nabla f|^2
			+
			2\alpha f^{\alpha+\frac{p}{2}-2} \left(|Hess\,\omega|^2 + Ric(\nabla \omega,\nabla \omega) \right)
			\\
			&
			+
			\alpha(p-2)(\alpha-1)f^{\alpha+\frac{p}{2}-4}\langle\nabla f,\nabla \omega\rangle^2
			+2\alpha f^{\alpha-1}\langle\nabla\Delta_p \omega,\nabla \omega\rangle
		\end{split}
	\end{align}
	holds point-wise  in $\{x:f(x)>0\}$.
\end{lem}
\begin{rmk}
	The transformation \eqref{equ 2.1} is an adjustment of that in {\rm\cite{YW}}. The case $b<0$, differing from previous scenarios, is motivated by studies on the fast diffusion equation in {\rm\cite{Zhu,HM}}.
	
\end{rmk}
\section{Preparation for linearization operator}\label{pr for linear p-laplace}
\subsection{Estimates for linearization operator of $p$-Laplacian}\label{section3.1}
In this section we  prove some lower bound estimates for $\mathcal L(f^\alpha)$ in different scenarios. Equations \eqref{equ1 for omega}, \eqref{equ3 for omega} and \eqref{equ2 for omega} will be considered respectively. 
 
 	Choose an orthonormal basis of $TM$  $\{e_1,e_2,\ldots, e_n\}$ on a domain with $f\neq 0$ such that $e_1=\frac{\nabla \omega}{|\nabla \omega|}$. We have $\omega_1 = f^{1/2}$ and
 \begin{align}\label{equ 3.1}
 	\omega_{11} = \frac{1}{2}f^{-1/2}f_1 = \frac{1}{2}f^{-1}\langle \nabla \omega,\nabla f\rangle .
 \end{align}
 
 Here $\omega_1$ represents the derivative of function $\omega$ along $e_{1}$ and $\omega_{11}$ is also similarly defined. Rewrite $p$-Laplace operator under this set of frames. From \cite{KN,WZ}, it has such an expression
 \begin{align*}
 	\Delta_p \omega
 	=&f^{\frac{p}{2}-1}\left((p-1)\omega_{11}+\sum_{i=2}^n\omega_{ii}\right).
 \end{align*}
 Substituting the above equality into equation \eqref{equ1 for omega}, we get 
 	\begin{equation}\label{equ 3.2}
 	(p-1)\omega_{11}+\sum_{i=2}^n\omega_{ii}=cf+de^{k\omega}f^{1-\frac{p}{2}}.
 \end{equation}
 Note that the following inequalities hold
 \begin{gather}
 	|\nabla f|^2/f=4\sum_{i=1}^n u_{1i}^2 \ge 4 \omega_{11}^{2}, \label{equ 3.3}\\
 	|Hess\,\omega|^2 
 	\geq \omega_{11}^2 +
 	\sum_{i=2}^{n}\omega_{ii}^2
 	\geq \omega_{11}^{2} +\frac{1}{n-1}\left(\sum_{i=2}^{n}\omega_{ii}\right)^2,\label{equ 3.4}
 \end{gather}
 where  formula \eqref{equ 3.4} is  gained from Cauchy-Schwarz's inequality.\\
 We begin with equation \eqref{equ1 for omega}, whose structure implies that
 \begin{align}\label{equ 3.5}
 	\langle\nabla\Delta_p \omega,\nabla \omega\rangle =  cpf^{\frac{p}{2} }\omega_{11}+kde^{k\omega} f .
 \end{align}
 Substituting \eqref{equ 3.1}, \eqref{equ 3.3}, \eqref{equ 3.4} and \eqref{equ 3.5} into  equality \eqref{lemma 2.2}, we derive

 \begin{align}\label{equ 3.6}
 	\begin{split}
 		\quad\frac{f^{2-\alpha-\frac{p}{2}}}{2\alpha} \mathcal{L} \left(f^{\alpha}\right)
 		\geq
 		&
 		(2\alpha-1)(p-1)\omega_{11}^{2}
 		+\frac{1}{n-1}\left(\sum_{i=2}\omega_{ii}\right)^2+ {\rm{Ric}}(\nabla \omega, \nabla \omega)
 		\\
 		&
 		+cpf\omega_{11}+kde^{k\omega}f^{2-\frac{p}{2}}.
 	\end{split}				
 \end{align}
 Meanwhile, from \eqref{equ 3.6} we have
 \begin{align}\label{equ 3.7}
 	\begin{split}
 			\frac{1}{n-1}\left(\sum_{i=2}^{n} \omega_{ii}\right)^2
 		=&   \frac{1}{n-1}\left(cf+de^{kw}f^{1-\frac{p}{2}}-(p-1)\omega_{11}\right)^2\\
 		=&\frac{c^{2}f^{2}}{n-1}+\frac{d^{2}e^{2k\omega}f^{2-p}}{n-1}+\frac{(p-1)^{2}\omega_{11}^{2}}{n-1}+\frac{2cde^{k\omega}f^{2-\frac{p}{2}}}{n-1}\\		
 		&-\frac{2c(p-1)f\omega_{11}}{n-1}-\frac{2d(p-1)e^{k\omega}f^{1-\frac{p}{2}}\omega_{11}}{n-1}.
 	\end{split}
 \end{align}

 Substituting \eqref{equ 3.7} into \eqref{equ 3.6} and using the condition ${\rm{Ric}}(\nabla \omega, \nabla \omega)\ge -(n-1)Kf$, we arrive at
 \begin{align}\label{equ 3.8}
 	\begin{split}
 		\frac{f^{2-\alpha-\frac{p}{2}}}{2\alpha} \mathcal{L} \left(f^{\alpha}\right)
 		\geq\;
 		&
 		\frac{c^{2}}{n-1}\,f^{2}
 		-(n-1)Kf+\frac{d^{2}}{n-1}\,e^{2k\omega}f^{2-p}+\left(cp-\frac{2c(p-1)}{n-1}\right)f\omega_{11}
 		\\
 		&+\left((p-1)(2\alpha-1)+\frac{(p-1)^{2}}{n-1}\right)\omega_{11}^2-\frac{2d(p-1)}{n-1}e^{k\omega}f^{1-\frac{p}{2}}\omega_{11}	
 		\\
 		&
 		+d(k+\frac{2c}{n-1})e^{k\omega}f^{2-\frac{p}{2}}.
 	\end{split}			
 \end{align}
 Now we handle the second line in \eqref{equ 3.8}. Using the inequality
 \[a^{2}-2ab\ge -b^{2},\]
  we arrive at
 \begin{align}
 	\label{equ 3.9}
 	\begin{split}
 		&\left((p-1)(2\alpha-1)+\frac{(p-1)^{2}}{n-1}\right)\omega_{11}^2-\frac{2d(p-1)}{n-1}e^{k\omega}f^{1-\frac{p}{2}}\omega_{11}	\\
 		\qquad&	\geq  
 		-\frac{d^{2}(p-1)}{(2\alpha-1)(n-1)^{2}+(p-1)(n-1)}e^{2k\omega}f^{2-p}.
 	\end{split}
 \end{align}
 Denote 
 \begin{equation}\label{parameter def1}	
 			\mu_{n,p,\alpha}\triangleq\frac{1}{n-1}-\frac{p-1}{(2\alpha-1)(n-1)^{2}+(p-1)(n-1)}.\\
 \end{equation}
 It is obvious that
 \[\mu_{n,p,\alpha}=\frac{(2\alpha-1)}{(2\alpha-1)(n-1)+(p-1)}\to \frac{1}{n-1},\quad \text{as}\;\,\alpha\to \infty.\]
  Substitute \eqref{equ 3.9} into \eqref{equ 3.8}, it yields
 \begin{align}\label{equ 3.11}
	\begin{split}
		\frac{f^{2-\alpha-\frac{p}{2}}}{2\alpha} \mathcal{L} \left(f^{\alpha}\right)
		\geq\;
		&
		\frac{c^{2}}{n-1}\,f^{2}+\mu_{n,p,\alpha}d^{2}e^{2k\omega}f^{2-p}
		-(n-1)Kf
		\\
		&
	+\left(cp-\frac{2c(p-1)}{n-1}\right)f\omega_{11}	+d(k+\frac{2c}{n-1})e^{k\omega}f^{2-\frac{p}{2}}.
	\end{split}			
\end{align}
Notice that the following formula holds, if we set $a_{1}=\left|cp-\frac{2c(p-1)}{n-1}\right|$,
\begin{equation}\label{equ 3.12}
	\left(cp-\frac{2c(p-1)}{n-1}\right)f\omega_{11}\ge -\frac{a_{1}}{2}f^{\frac{1}{2}}|\nabla f|.
\end{equation}
A combination of \eqref{equ 3.11} and \eqref{equ 3.12} reaches
 \begin{align}\label{equ 3.13}
	\begin{split}
		\frac{f^{2-\alpha-\frac{p}{2}}}{2\alpha} \mathcal{L} \left(f^{\alpha}\right)
		\geq\;
		&
		\frac{c^{2}}{n-1}\,f^{2}+\mu_{n,p,\alpha}d^{2}e^{2k\omega}f^{2-p}
		-(n-1)Kf
		\\
		&
		-\frac{a_{1}}{2}f^{\frac{1}{2}}|\nabla f|	+d(k+\frac{2c}{n-1})e^{k\omega}f^{2-\frac{p}{2}}.
	\end{split}			
\end{align}
 \textbf{Case I:} If the last term in \eqref{equ 3.13} is nonnegative, i.e.
 \begin{equation}\label{condition I}
 	d(k+\frac{2c}{n-1})e^{k\omega}f^{2-\frac{p}{2}}\ge 0,
 \end{equation}
 by omitting some nonnegative terms in \eqref{equ 3.13} we arrive at
 \begin{equation}\label{esitmate1}
 		\frac{f^{2-\alpha-\frac{p}{2}}}{2\alpha} \mathcal{L} \left(f^{\alpha}\right)	\geq	
 	\frac{c^{2}}{n-1}\,f^{2}
 	-(n-1)Kf
 	-\frac{a_{1}}{2}f^{\frac{1}{2}}|\nabla f|.	
 \end{equation}
 \textbf{Case II:} Via the inequality $a^{2}+2ab\ge -b^{2}$, we have
 \begin{equation}\label{equ 3.16}
 	\mu_{n,p,\alpha}d^{2}e^{2k\omega}f^{2-p}+	d(k+\frac{2c}{n-1})e^{k\omega}f^{2-\frac{p}{2}}\ge -\frac{1}{4\mu}\left(k+\frac{2c}{n-1}\right)^{2}f^{2}.
 \end{equation}
  By coupling \eqref{equ 3.16} and \eqref{equ 3.11}, it provides 
  \begin{equation}\label{estimate2}
  		\frac{f^{2-\alpha-\frac{p}{2}}}{2\alpha} \mathcal{L} \left(f^{\alpha}\right)	\geq	
  		  \sigma_{1} f^{2}
  		-(n-1)Kf
  		-\frac{a_{1}}{2}f^{\frac{1}{2}}|\nabla f|,
  \end{equation}
  where
  \begin{equation}\label{condition II}
  	\sigma_{1}=\sigma_{1}(n,p,q,\gamma,\alpha)\triangleq \frac{c^{2}}{n-1}-\frac{1}{4\mu}\left(k+\frac{2c}{n-1}\right)^{2}.
  \end{equation}
  
  The treatment of  equation \eqref{equ3 for omega} is only a minor adaptation of the process described above. Similar to \eqref{equ 3.13}, we have
  \begin{align}\label{equ 3.19}
  	\begin{split}
  		\frac{f^{2-\alpha-\frac{p}{2}}}{2\alpha} \mathcal{L} \left(f^{\alpha}\right)
  		\geq\;
  		&
  		\frac{c^{2}}{n-1}\,f^{2}+\mu_{n,p,\alpha}l^{2}e^{2k\omega}f^{2-p}
  		-(n-1)Kf
  		\\
  		&
  		-\frac{a_{1}}{2}f^{\frac{1}{2}}|\nabla f|	-l(k+\frac{2c}{n-1})e^{k\omega}f^{2-\frac{p}{2}}.
  	\end{split}			
  \end{align}
  \textbf{Case III:} If the last term in \eqref{equ 3.19} is non-positive, i.e.
  
   \begin{equation}\label{condition III}
  	l(k+\frac{2c}{n-1})e^{k\omega}f^{2-\frac{p}{2}}\le 0,
  \end{equation}
  by omitting some nonnegative terms in \eqref{equ 3.19} we also arrive at 
  \begin{equation}\label{esitmate3}
  	\frac{f^{2-\alpha-\frac{p}{2}}}{2\alpha} \mathcal{L} \left(f^{\alpha}\right)	\geq	
  	\frac{c^{2}}{n-1}\,f^{2}
  	-(n-1)Kf
  	-\frac{a_{1}}{2}f^{\frac{1}{2}}|\nabla f|.	
  \end{equation}
  \textbf{Case IV:} We use the inequality $a^{2}-2ab\ge -b^{2}$ to obtain 
   \begin{equation}\label{equ 3.22}
  	\mu_{n,p,\alpha}l^{2}e^{2k\omega}f^{2-p}-l	(k+\frac{2c}{n-1})e^{k\omega}f^{2-\frac{p}{2}}\ge -\frac{1}{4\mu}\left(k+\frac{2c}{n-1}\right)^{2}f^{2}.
  \end{equation}
  The same discussion as in \textbf{Case II} leads to 

  \begin{equation}\label{estimate4}
  	\frac{f^{2-\alpha-\frac{p}{2}}}{2\alpha} \mathcal{L} \left(f^{\alpha}\right)	\geq	
  	\sigma_{1} f^{2}
  	-(n-1)Kf
  	-\frac{a_{1}}{2}f^{\frac{1}{2}}|\nabla f|.
  \end{equation}

  The behavior up to equation \eqref{equ2 for omega} is similar but with minor changes. Instead of \eqref{equ 3.2}, \eqref{equ 3.5} and \eqref{equ 3.7}, we replace each by
 \begin{equation}
	(p-1)\omega_{11}+\sum_{i=2}^n\omega_{ii}=-(p-1)f-ae^{(p-1)(q-1)\omega}f^{1-\frac{p}{2}},
 \end{equation}
 \begin{equation}
 	\langle\nabla\Delta_p \omega,\nabla \omega\rangle= -\frac{p(p-1)}{2}f^{\frac{p}{2}-1}\langle \nabla f,\nabla \omega \rangle-a(p-1)(q-1)e^{(p-1)(q-1)\omega}f^{1-\frac{p}{2}},
 \end{equation}
 and
 \begin{align}
 	\begin{split}
 		\frac{1}{n-1}\left(\sum_{i=2}^{n} \omega_{ii}\right)^2
 		=&   \frac{1}{n-1}\left((p-1)\omega_{11}+(p-1)f+a(p-1)(q-1)e^{(p-1)(q-1)\omega}f^{1-\frac{p}{2}}\right)^2\\
 		=&\frac{(p-1)^{2}f^{2}}{n-1}+\frac{a^{2}e^{2(p-1)(q-1)\omega}f^{2-p}}{n-1}+\frac{(p-1)^{2}\omega_{11}^{2}}{n-1}+\frac{2(p-1)^{2}f\omega_{11}}{n-1}\\		
 		&+\frac{2a(p-1)e^{(p-1)(q-1)\omega}f^{2-\frac{p}{2}}}{n-1}+\frac{2a(p-1)e^{(p-1)(q-1)\omega}f^{1-\frac{p}{2}}\omega_{11}}{n-1}.
 	\end{split}
 \end{align}
Proceeding in a similar manner, we easily get
\begin{align}\label{equ 3.27}
	\begin{split}
		\frac{f^{2-\alpha-\frac{p}{2}}}{2\alpha} \mathcal{L} \left(f^{\alpha}\right)
		\geq\;
		&
		\frac{(p-1)^{2}}{n-1}\,f^{2}
		-(n-1)Kf+\frac{a^{2}}{n-1}\,e^{2(p-1)(q-1)\omega}f^{2-p}-\frac{a_{1}}{2}f^{\frac{1}{2}}|\nabla f|
		\\
		&+\left((p-1)(2\alpha-1)+\frac{(p-1)^{2}}{n-1}\right)\omega_{11}^2+\frac{2a(p-1)}{n-1}e^{(p-1)(q-1)\omega}f^{1-\frac{p}{2}}\omega_{11}	
		\\
		&
		+\left(\frac{2a(p-1)}{n-1}-a(p-1)(q-1)\right)e^{(p-1)(q-1)\omega}f^{2-\frac{p}{2}}.
	\end{split}			
\end{align}
Substituting the inequality
\begin{align*}
	\begin{split}
	&\left((p-1)(2\alpha-1)+\frac{(p-1)^{2}}{n-1}\right)\omega_{11}^2+
	\frac{2a(p-1)}{n-1}e^{(p-1)(q-1)\omega}f^{1-\frac{p}{2}}\omega_{11}	\\
		\qquad&	\geq  
		-\frac{a^{2}(p-1)f^{2-p}}{(2\alpha-1)(n-1)^{2}+(p-1)(n-1)}e^{2(p-1)(q-1)\omega}
	\end{split}
\end{align*}
into \eqref{equ 3.27}, it provides
 \begin{align*}
	\begin{split}
		\frac{f^{2-\alpha-\frac{p}{2}}}{2\alpha} \mathcal{L} \left(f^{\alpha}\right)
		\geq\;
		&
		\frac{(p-1)^{2}}{n-1}\,f^{2}+a^{2}\mu_{n,p,\alpha}e^{2(p-1)(q-1)\omega}f^{2-p}
		-(n-1)Kf\\	
		&
		+\left(\frac{2a(p-1)}{n-1}-a(p-1)(q-1)\right)e^{(p-1)(q-1)\omega}f^{2-\frac{p}{2}}.\\
			&+\left(\frac{2(p-1)^{2}}{n-1}-p(p-1)\right)f\omega_{11}.
	\end{split}			
\end{align*} 
Analogously, we set  $a_{2}=\left|\frac{2(p-1)^{2}}{n-1}-p(p-1)\right|$  and also get
\begin{equation*}
	\left(\frac{2(p-1)^{2}}{n-1}-p(p-1)\right)f\omega_{11}\ge -\frac{a_{2}}{2}f^{\frac{1}{2}}|\nabla f|.
\end{equation*}
Similar to \eqref{equ 3.13}, we derive
 \begin{align}\label{equ 3.28}
	\begin{split}
		\frac{f^{2-\alpha-\frac{p}{2}}}{2\alpha} \mathcal{L} \left(f^{\alpha}\right)
		\geq\;
		&
		\frac{(p-1)^{2}}{n-1}\,f^{2}+a^{2}\mu_{n,p,\alpha}e^{2(p-1)(q-1)\omega}f^{2-p}
		-(n-1)Kf-\frac{a_{2}}{2}f^{\frac{1}{2}}|\nabla f|\\	
		&
		+\left(\frac{2a(p-1)}{n-1}-a(p-1)(q-1)\right)e^{(p-1)(q-1)\omega}f^{2-\frac{p}{2}}.
	\end{split}			
\end{align} 
\textbf{Case V:} If the last term in \eqref{equ 3.28} is nonnegative, i.e. 
\begin{equation}\label{condition V}
	\frac{2a(p-1)}{n-1}-a(p-1)(q-1)\ge 0,
\end{equation} 
then by discarding some nonnegative terms in \eqref{equ 3.28}, we arrive at
  \begin{equation}\label{esitmate5}
 	\frac{f^{2-\alpha-\frac{p}{2}}}{2\alpha} \mathcal{L} \left(f^{\alpha}\right)	\geq	
 	\frac{(p-1)^{2}}{n-1}\,f^{2}
 	-(n-1)Kf
 	-\frac{a_{2}}{2}f^{\frac{1}{2}}|\nabla f|.	
 \end{equation}
 \textbf{Case VI:} Following the same procedure in \textbf{Case II}, we find

 \begin{align*}
 	\begin{split}
 		 a^{2}\mu_{n,p,\alpha}e^{2(p-1)(q-1)\omega}f^{2-p}&+\left(\frac{2a(p-1)}{n-1}-a(p-1)(q-1)\right)e^{(p-1)(q-1)\omega}f^{2-\frac{p}{2}}	\\
 		\qquad&	\geq  
 	-\frac{1}{4\mu}\left(\frac{2(p-1)}{n-1}-(p-1)(q-1)\right)^{2}f^{2}.
 	\end{split}
 \end{align*}
 Combining the above formula and \eqref{equ 3.28}, we obtain 
 \begin{equation}\label{estimate6}
	\frac{f^{2-\alpha-\frac{p}{2}}}{2\alpha} \mathcal{L} \left(f^{\alpha}\right)	\geq	
	\sigma_{2} f^{2}
	-(n-1)Kf
	-\frac{a_{2}}{2}f^{\frac{1}{2}}|\nabla f|,
\end{equation}
where
\begin{equation}\label{condition VI}
	\sigma_{2}=\sigma_{2}(n,p,q,\gamma,\alpha)\triangleq (p-1)^{2}\left[\frac{1}{n-1}-\frac{1}{4\mu}\left(\frac{2}{n-1}-(q-1) \right)^{2}\right].
\end{equation}

To maintain consistency with $\sigma_{1}$, we retain the dependence of $\sigma_{2}$ on $\gamma$ despite its prior determination.
 \subsection{A further discussion about coefficients}\label{section3.2} We deeply discuss the conditions mentioned in \secref{section3.1} and derive the relationship that  parameters $a,n,p,q$ and $\gamma$ satisfy.
 First of all, we consider  \textbf{Case I} and \textbf{Case III}.
By combining condition \eqref{condition I} with the explicit expressions for 
$c,d,k$ and $l$, we derive
  \begin{equation}
  	   a(p-1)^{p-1} \left(\frac{b}{\gamma}\right)^{\frac{q-1}{b}}\left[\frac{n+1}{n-1}+\frac{2-(n-1)(q-1)}{b(n-1)(p-1)}\right]\ge 0.
  \end{equation}
  Since $b/\gamma$ keeps positive for $\gamma>\frac{1}{p-1}$, this condition reduces to
  \begin{equation}\label{cd1}
 	a\left[\frac{n+1}{n-1}+\frac{2-(n-1)(q-1)}{b(n-1)(p-1)}\right]\ge 0.
  \end{equation}
  By a straightforward calculation, we know
  $$ \frac{n+1}{n-1}+\frac{2-(n-1)(q-1)}{b(n-1)(p-1)}=\frac{\gamma(n+1)(p-1)-(n-1)q}{(n-1)\left(\gamma(p-1)-1\right)}.$$
  Utilizing the above formula and solving inequality \eqref{cd1}, we derive
  \begin{equation}\label{parameter restriction1}
  	 \left\{\begin{matrix}
  	 	\begin{aligned}
  	 		a&\ge 0,\\
  	 		\gamma&>\frac{1}{p-1},\\
  	 		p&\ge 1+\frac{(n-1)q}{(n+1)\gamma},
  	 	\end{aligned}	
  	\end{matrix}\right.
  	\qquad  or \qquad 
  	 \left\{\begin{matrix}
  		\begin{aligned}
  			a&\le 0,\\
  			\gamma&>\frac{1}{p-1},\\
  			p&\le 1+\frac{(n-1)q}{(n+1)\gamma}.
  		\end{aligned}	
  	\end{matrix}\right.
  \end{equation}
  
  Noting that $-b/\gamma$ is still positive for $0<\gamma<\frac{1}{p-1}$, we apply the same argument to \eqref{condition III} and get
\begin{equation}\label{cd2}
	a\left[\frac{n+1}{n-1}+\frac{2-(n-1)(q-1)}{b(n-1)(p-1)}\right]\le 0.
\end{equation}
 By solving  inequality \eqref{cd2}, we derive
 \begin{equation}\label{parameter restriction2}
  \left\{\begin{matrix}
 	\begin{aligned}
 		a&\ge 0,\\
 		0&<\gamma<\frac{1}{p-1},\\
 		p&\ge 1+\frac{(n-1)q}{(n+1)\gamma},
 	\end{aligned}	
 \end{matrix}\right.
 \qquad  or \qquad 
 \left\{\begin{matrix}
 	\begin{aligned}
 		a&\le 0,\\
 		0&<\gamma<\frac{1}{p-1},\\
 		p&\le 1+\frac{(n-1)q}{(n+1)\gamma}.
 	\end{aligned}	
 \end{matrix}\right.
 \end{equation}
Towards \textbf{Case V}, by  solving \eqref{condition V} directly, it suggests 
 \begin{equation}\label{parameter restricition3}
 		\left\{\begin{matrix}
 		\begin{aligned}
 			a&\ge 0,\\
 			q&\le\frac{n+1}{n-1},
 		\end{aligned}	
 	\end{matrix}\right.
 	\quad  or \quad
 	\left\{\begin{matrix}
 		\begin{aligned}
 			a&\le 0,\\
 		q&\ge\frac{n+1}{n-1}.
 		\end{aligned}	
 	\end{matrix}\right.
 \end{equation}
 
 	We now analyze the conditions under which  the coefficients $\sigma_{1}$ and $\sigma_{2}$ are positive in \textbf{Case II}, \textbf{Case IV} and \textbf{Case VI}, respectively. Recall that $\mu_{n,p,\alpha}$ tends to $1/(n-1)$ as $\alpha$ tends to infinity. Henceforth, if the following condition holds
 	\begin{equation}\label{inequ for sigma1}
 		\frac{c^{2}}{n-1}-\frac{n-1}{4}\left(k+\frac{2c}{n-1}\right)^{2}>0,
 	\end{equation}
 	then there exists some $\alpha_{1}\ge \frac{3}{2}$ depending on $n,p,q,\gamma$ such that $\sigma_{1}>0$  as long as $\alpha\ge \alpha_{1}$. Solving inequality \eqref{inequ for sigma1}, we have
 	 \begin{equation}\label{equ 3.40}
 		\left\{\begin{matrix}
 			\begin{aligned}
 				k&> 0,\\
 				k&+\frac{4c}{n-1}<0,
 			\end{aligned}	
 		\end{matrix}\right.
 		\quad  or \quad
 		\left\{\begin{matrix}
 			\begin{aligned}
 				k&< 0,\\
 				k&+\frac{4c}{n-1}>0.
 			\end{aligned}	
 		\end{matrix}\right.
 	\end{equation}
 	Substitute the expression of $c$ and $k$ into \eqref{equ 3.40}, the first group of inequalities become
 	\begin{gather}\label{1st}
 		\begin{cases}
 			1-\frac{q-1}{\gamma(p-1)-1}>0,\\
 			1-\frac{q-1}{\gamma(p-1)-1}+\frac{4}{n-1}\left(1+\frac{1}{\gamma(p-1)-1}\right)<0.
 		\end{cases}
 	\end{gather}
	Since \eqref{1st} admits no solution when $\gamma>\frac{1}{p-1}$, we get
	$$  \gamma(p-1)<q<\frac{n+3}{n-1}\gamma(p-1),\quad \;0<\gamma<\frac{1}{p-1}.$$
 	Meanwhile, the second group of inequalities in \eqref{equ 3.40} imply
 		\begin{gather}\label{2nd}
 		\begin{cases}
 			1-\frac{q-1}{\gamma(p-1)-1}<0,\\
 			1-\frac{q-1}{\gamma(p-1)-1}+\frac{4}{n-1}\left(1+\frac{1}{\gamma(p-1)-1}\right)>0.
 		\end{cases}
 	\end{gather}
 	It follows from \eqref{2nd} analogously that 
 	\begin{equation*}
 		\gamma(p-1)<q<\frac{n+3}{n-1}\gamma(p-1),\quad \;\gamma>\frac{1}{p-1}. 
 	\end{equation*}
 	Combining these mentioned above,  \eqref{inequ for sigma1} implies 
 	\begin{equation}\label{need1}
 		 \gamma(p-1)<q<\frac{n+3}{n-1}\gamma(p-1),\quad \;\gamma\neq\frac{1}{p-1}. 
 	\end{equation}
 	Similarly, $\sigma_{2}$ stays positive provided the following condition holds
	\begin{equation}\label{inequ for sigma2}
	\frac{1}{n-1}-\frac{n-1}{4}\left(\frac{2}{n-1}-(q-1) \right)^{2}>0.
	\end{equation}
 	It solves 
 	\begin{equation}\label{need2}
 		1<q<\frac{n+3}{n-1},
 	\end{equation}
 	where in this case $\gamma=1/(p-1)$. Consequently, we also confirm that there exists some $\alpha_{2}\ge \frac{3}{2}$ relying on $n,p,q,\gamma$ such that $\sigma_{2}>0$ if $\alpha\ge \alpha_{2}$.
 	
 \subsection{Some technical lemmas}
 From the discussion in \secref{section3.1} and \secref{section3.2}, we can obtain the following lemmas.
 \begin{lem}\label{lem3.1}
 	Let $(M^{n},g)$ be a complete Riemannian manifold satisfying $Ric(M)\ge -(n-1)K g$ for some constant $K\ge 0$. Denote $b=\gamma-\tfrac{1}{p-1}>0$ and $c=1+\frac{1}{b(p-1)}$. If parameters $a,p,q,\gamma$ and $n$ satisfy
 	\begin{equation*}
 		a\left[\frac{n+1}{n-1}+\frac{2-(n-1)(q-1)}{b(n-1)(p-1)}\right]\ge 0,
 	\end{equation*}
 	 then the following estimate holds
 	\[\mathcal{L} (f^{\alpha})\geq\frac{2\alpha c^{2}}{n-1}\,f^{\alpha+\frac{p}{2}}-2\alpha(n-1)Kf^{\alpha+\frac{p}{2}-1}-a_{1}\alpha f^{\alpha+\frac{p}{2}-\frac{3}{2}}|\nabla f|.\]
 \end{lem}
 \begin{lem}\label{lem3.2}
 			Let $(M^{n},g)$ be a complete Riemannian manifold satisfying $Ric(M)\ge -(n-1)K g$ for some constant $K\ge 0$.  Denote $b=\gamma-\tfrac{1}{p-1}\neq 0$ and $c=1+\frac{1}{b(p-1)}$. If the following condition
 		$$ \gamma(p-1)<q<\frac{n+3}{n-1}\gamma(p-1)$$
 	   holds for any $a\in \BR$, then there exists a constant $\alpha_{1}\ge \frac{3}{2}$ depending on $n,p,q,\gamma$ such that for  $\alpha\ge \alpha_{1}$,
 		\[ 	 \mathcal{L} \left(f^{\alpha}\right)	\geq	
 		2\alpha \sigma_{1} f^{\alpha+\frac{p}{2}}
 		-2\alpha(n-1)Kf^{\alpha+\frac{p}{2}-1}-a_{1}\alpha f^{\alpha+\frac{p}{2}-\frac{3}{2}}|\nabla f|,\]
 		where $\sigma_{1}$ is defined in \eqref{condition II}.
 \end{lem}
 \begin{lem}\label{lem3.3}
 		Let $(M^{n},g)$ be a complete Riemannian manifold satisfying $Ric(M)\ge -(n-1)K g$ for some constant $K\ge 0$.   Denote $\gamma=\frac{1}{p-1}$, i.e. $b=0$. If parameters $a,n,p$ and $q$ satisfy
 		$$	\frac{2a(p-1)}{n-1}-a(p-1)(q-1)\ge 0,$$
 	 then the following estimate holds
 	\[\mathcal{L} (f^{\alpha})\geq\frac{2\alpha (p-1)^{2}}{n-1}\,f^{\alpha+\frac{p}{2}}-2\alpha(n-1)Kf^{\alpha+\frac{p}{2}-1}-a_{2}\alpha f^{\alpha+\frac{p}{2}-\frac{3}{2}}|\nabla f|.\]
 \end{lem}
 \begin{lem}\label{lem3.4}
 		Let $(M^{n},g)$ be a complete Riemannian manifold satisfying $Ric(M)\ge -(n-1)K g$ for some constant $K\ge 0$. Denote $\gamma=\frac{1}{p-1}$, i.e. $b=0$. If the following condition 
 			$$1<q<\frac{n+3}{n-1} $$
 	 holds for any $a\in \BR$, then there exists a constant $\alpha_{2}\ge \frac{3}{2}$ depending on $n,p,q,\gamma$ such that for  $\alpha\ge \alpha_{2}$,
 	\[ 	 \mathcal{L} \left(f^{\alpha}\right)	\geq	
 2\alpha \sigma_{2} f^{\alpha+\frac{p}{2}}
 -2\alpha(n-1)Kf^{\alpha+\frac{p}{2}-1}-a_{2}\alpha f^{\alpha+\frac{p}{2}-\frac{3}{2}}|\nabla f|,\]
 where $\sigma_{2}$ is given by \eqref{condition VI}.
 \end{lem}
  \begin{lem}\label{lem3.5}
 	Let $(M^{n},g)$ be a complete Riemannian manifold satisfying $Ric(M)\ge -(n-1)K g$ for some constant $K\ge 0$. Denote $b=\gamma-\tfrac{1}{p-1}< 0$ and $c=1+\frac{1}{b(p-1)}$. If parameters $a,p,q,\gamma$ and $n$ satisfy
 	\begin{equation*}
 		a\left[\frac{n+1}{n-1}+\frac{2-(n-1)(q-1)}{b(n-1)(p-1)}\right]\le 0,
 	\end{equation*}
 	then the following estimate holds
 	\[\mathcal{L} (f^{\alpha})\geq\frac{2\alpha c^{2}}{n-1}\,f^{\alpha+\frac{p}{2}}-2\alpha(n-1)Kf^{\alpha+\frac{p}{2}-1}-a_{1}\alpha f^{\alpha+\frac{p}{2}-\frac{3}{2}}|\nabla f|.\]
 \end{lem}
 \begin{proof}[Proof of \lemref{lem3.1}:]
 The assertion is an immediate consequence of \eqref{condition I}, \eqref{esitmate1} and \eqref{cd1} .
 \end{proof}
 \begin{proof}[Proof of \lemref{lem3.2}:]
 	It is a direct deduction from \eqref{estimate2}, \eqref{condition II}, \eqref{estimate4}, \eqref{inequ for sigma1} and \eqref{need1}. 
 \end{proof}
 \begin{proof}[Proof of \lemref{lem3.3}:]
  It easily follows from \eqref{condition V} and \eqref{esitmate5}.
 \end{proof}
 \begin{proof}[Proof of \lemref{lem3.4}:]
 	A combination of \eqref{estimate6}, \eqref{condition VI}, \eqref{inequ for sigma2} and \eqref{need2}  asserts it.
 \end{proof}
 \begin{proof}[Proof of \lemref{lem3.5}:]
 	The assertion is an immediate consequence of  \eqref{condition III}, \eqref{esitmate3}  and \eqref{cd2}.
 \end{proof}

 \section{$L^{\infty}$ bound of $|\nabla \omega|^{2}$ } \label{standard section}
 \subsection{Integral inequality}
 In  section \ref{pr for linear p-laplace}, we have derived five lemmas. For the sake of convenience, we present them in the following standardized format
 \begin{equation}\nonumber
 	\mathcal{L} (f^{\alpha})\geq  2\alpha\sigma f^{2}-2\alpha(n-1)Kf^{\alpha+\frac{p}{2}-1}-a_{3}\alpha f^{\alpha+\frac{p}{2}-\frac{3}{2}}|\nabla f|.
 \end{equation}
	
	In the five lemmas above, $\sigma = \sigma(n,p,q,\gamma,\alpha)$  takes the value $\frac{c^{2}}{n-1}$ in \lemref{lem3.1} and \lemref{lem3.5}, $\sigma_{1}$ in \lemref{lem3.2}, $\frac{(p-1)^{2}}{n-1}$ in \lemref{lem3.3}, and $\sigma_{2}$ in \lemref{lem3.4}.  The constant $a_{3}$ equals $a_{1}$ in  lemmas \ref{lem3.1}, \ref{lem3.2} and \ref{lem3.5}, and $a_{2}$ in lemmas \ref{lem3.3} and \ref{lem3.4}.
   Meanwhile, we choose $\alpha_{0}=\max\left\{\alpha_{1},\alpha_{2},\frac{3}{2}\right\}$ to guarantee that the above formula is valid for any $\alpha \ge \alpha_{0}$. Our aim is to give an integral inequality of $f$. For convenience, we fix $\alpha=\alpha_{0}$  in the following part and obtain
 \begin{equation}\label{lemma for integral}
 	\mathcal{L} (f^{\alpha_{0}})\geq  2\alpha_{0}\sigma f^{2}-2\alpha_{0}(n-1)Kf^{\alpha_{0}+\frac{p}{2}-1}-a_{3}\alpha_{0} f^{\alpha_{0}+\frac{p}{2}-\frac{3}{2}}|\nabla f|.
 \end{equation}
 
 The following step is a standard procedure. For the readers' convenience, we sketch the proof.  Firstly, we choose a geodesic ball $\Omega= B(o,R)\subset M$ and select the test function $\psi$ as follows
 \begin{equation*}
 	\psi=  f_{\epsilon}^{t}\,\eta^2,
 \end{equation*}
 where $\eta\in C^{\infty}_0(\Omega,\BR)$ is non-negative and $f_\epsilon=(f-\epsilon)^+$ with respect to some $\epsilon>0$. $t$ is greater than 1 and will be determined later. Integrate \eqref{lemma for integral} over 
 the region $\Omega$, there holds
 \begin{subequations}\label{equ 4.2}
 	\begin{align}	
 		&-\int_{\Omega} \alpha_{0} tf^{\alpha_{0}+\frac{p}{2}-2}f^{t-1}_{\epsilon}|\nabla f|^2\eta^2
 		+
 		\alpha_{0} t(p-2)f^{\alpha_{0}+\frac{p}{2}-3}f^{t-1}_\epsilon \langle \nabla f,\nabla \omega\rangle^{2}
 		\eta^{2}\notag
 		\\
 		&-\int_{\Omega} 2\alpha_{0}\eta f^{\alpha_{0}+\frac{p}{2}-2}f^{t}_\epsilon \langle \nabla f,\nabla\eta \rangle +2\alpha_{0}\eta(p-2)f^{\alpha_{0}+\frac{p}{2}-3}f^{t}_\epsilon \langle \nabla f,\nabla \omega\rangle \langle  \nabla \omega,  \nabla\eta\rangle \tag{4.2}  \\ 
 		\geq & 2\alpha_{0}\sigma\int_{\Omega} f^{\alpha_{0}+\frac{p}{2} }f^{t}_{\epsilon}\eta^2  -2\alpha_{0}(n-1)K\int_{\Omega} f^{\alpha_{0}+\frac{p}{2} -1}f^{t}_\epsilon\eta^2- a_3\alpha_{0}\int_{\Omega} f^{\alpha_{0}+\frac{p}{2}-\frac{3}{2} }f^{t}_\epsilon|\nabla f|\eta^2. \notag
 	\end{align}
 \end{subequations}		
 In order to handle terms involving inner product, we use such inequalities
 \begin{equation}\label{equ 4.3}
 	f^{t-1}_\epsilon |\nabla f|^2 +(p-2)f^{t-1}_\epsilon f^{-1}\langle\nabla f,\nabla \omega\rangle^2\geq c_{1} f^{t-1}_\epsilon|\nabla f|^2,
 \end{equation}
 \begin{equation}\label{equ 4.4}
 	f^{t}_\epsilon\langle\nabla f,\nabla\eta\rangle+ (p-2)f^{t}_\epsilon f^{-1}\langle\nabla f,\nabla \omega\rangle\langle \nabla \omega,  \nabla\eta\rangle \geq -(p+1)  f^{t}_\epsilon|\nabla f||\nabla\eta|,
 \end{equation}	

 where\, $c_{1}=\min\left\{1,p-1\right\}$. Thus we have 		

 \begin{align}\label{equ 4.5}
 	\begin{split}
 		-\int_{\Omega} \alpha_{0} tf^{\alpha_{0}+\frac{p}{2}-2}f^{t-1}_{\epsilon}|\nabla f|^2\eta^2
 		&+
 		\alpha_{0} t(p-2)f^{\alpha_{0}+\frac{p}{2}-3}f^{t-1}_\epsilon \langle \nabla f,\nabla \omega\rangle^{2}
 		\eta^{2} \\ 
 		&\le 	-\int_{\Omega} \alpha_{0} t\, c_{1} f^{\alpha_{0}+\frac{p}{2}-2} f^{t-1}_{\epsilon} \eta^{2} 
 		|\nabla f|^{2},
 	\end{split}
 \end{align}
 and
 \begin{align}\label{equ 4.6}
 	\begin{split}
 		-\int_{\Omega} 2\alpha_{0}\eta f^{\alpha_{0}+\frac{p}{2}-2}f^{t}_\epsilon \langle \nabla f,\nabla\eta \rangle &+2\alpha_{0}\eta(p-2)f^{\alpha_{0}+\frac{p}{2}-3}f^{t}_\epsilon \langle \nabla f,\nabla \omega\rangle \langle  \nabla \omega,  \nabla\eta\rangle \\
 		&\le 	\int_{\Omega} 2(p+1) \alpha_{0} \eta  f^{\alpha_{0}+\frac{p}{2}-2} f^{t}_{\epsilon} \eta
 		|\nabla f| |\nabla \eta|.
 	\end{split}
 \end{align}
 Substitute \eqref{equ 4.5} and \eqref{equ 4.6} into the formula \eqref{equ 4.2}, divide both sides by $\alpha_{0}$ and let $\epsilon$ tend to  zero, a straightforward computation shows that
 \begin{align}\label{equ 4.7}
 	\begin{split}
 		&2\sigma\int_{\Omega} f^{\alpha_{0}+\frac{p}{2}+t}\eta^2
 		+
 		c_{1}  t\int_{\Omega} f^{\alpha_{0}+\frac{p}{2}+t-3}|\nabla f|^2\eta^2\\
 		\leq &
 		2(n-1) K\int_{\Omega} f^{\alpha_{0}+\frac{p}{2}+t-1}\eta^2
 		+ a_3 \int_{\Omega} f^{\alpha_{0}+\frac{p-3}{2}+t }|\nabla f|\eta^2
 		\\
 		&
 		+2 (p+1)\int_{\Omega}  f^{\alpha_{0}+\frac{p}{2}+t-2}|\nabla f||\nabla\eta|\eta.
 	\end{split}
 \end{align}
 As a consequence of well-known Cauchy-Schwarz's inequality, we deduce 
 \begin{equation}
 	a_3 f^{\alpha_{0}+\frac{p-3}{2}+t }|\nabla f|\eta^2\leq
 	\frac{c_{1}t}{4}    f^{\alpha_{0}+\frac{p}{2}+t-3}|\nabla f|^2\eta^2
 	+\frac{ a_3^2}{c_{1}t} f^{\alpha_{0}+\frac{p}{2}+t}\eta^2
 \end{equation}
 and
 \begin{equation}
 	2(p+1)f^{\alpha_{0}+\frac{p}{2}+t-2}|\nabla f||\nabla\eta|\eta\leq
 	\frac{c_{1}t}{4}    f^{\alpha_{0}+\frac{p}{2}+t-3}|\nabla f|^2\eta^2
 	+\frac{4(p+1)^2 }{c_{1}t} f^{\alpha_{0}+\frac{p}{2}+t-1}|\nabla \eta|^2.
 \end{equation}
 We choose $t$ large enough to ensure 
 \begin{equation}\label{equ 4.10}
 	\frac{a_3^2}{c_{1}t}\leq  \sigma.
 \end{equation}
 Combining \eqref{equ 4.7}-\eqref{equ 4.10}, we can obtain
 \begin{align}
 	\label{equ 4.11}
 	\begin{split}
 		& \sigma\int_{\Omega} f^{\alpha_{0}+\frac{p}{2}+t}\eta^2
 		+
 		\frac{c_{1}  t}{2}\int_{\Omega} f^{\alpha_{0}+\frac{p}{2}+t-3}|\nabla f|^2\eta^2 \\
 		\leq  &
 		2(n-1) K\int_{\Omega} f^{\alpha_{0}+\frac{p}{2}+t-1}\eta^2
 		+\frac{4(p+1)^2 }{c_{1}t}  \int_{\Omega} f^{\alpha_{0}+\frac{p}{2}+t-1}|\nabla \eta|^2.
 	\end{split}
 \end{align}
 There is a fact that 
 \begin{align}\nonumber
 	\begin{split}
 		\left|\nabla \left(f^{\frac{\alpha_{0}+t-1}{2}+\frac{p}{4} }\eta \right)\right|^2\leq & 2\left|\nabla f^{\frac{\alpha_{0}+t-1}{2}+\frac{p}{4} }\right|^2\eta^2 +2f^{\alpha_{0}+t-1+\frac{p}{2}}|\nabla\eta|^2
 		\\
 		=&\frac{(2\alpha_{0}+2t+p-2)^2}{8}f^{\alpha_{0}+t+\frac{p}{2}-3}|\nabla f |^2\eta^2 +2f^{\alpha_{0}+t-1+\frac{p}{2}}|\nabla\eta|^2 .
 	\end{split}
 \end{align}
 It is equivalent to
 \begin{align}\label{equ 4.12}
 	\begin{split}
 		&\frac{4c_{1}t}{(2\alpha_{0}+2t+p-2)^{2}}\int_{\Omega}\left|\nabla \left(f^{\frac{\alpha_{0}+t-1}{2}+\frac{p}{4} }\eta \right)\right|^2   \\ &\leq\frac{c_{1}t}{2}\int_{\Omega}f^{\alpha_{0}+t+\frac{p}{2}-3}|\nabla f |^2\eta^2
 		+\frac{8c_{1}t}{(2\alpha_{0}+2t+p-2)^{2}}\int_{\Omega} f^{\alpha_{0}+t-1+\frac{p}{2}}|\nabla\eta|^2.
 	\end{split}
 \end{align}
 Substituting \eqref{equ 4.12} into \eqref{equ 4.11}, we get
 \begin{align}\label{equ 4.13}
 	\begin{split}
 		\sigma \int_{\Omega} f^{\alpha_{0}+\frac{p}{2}+t}\eta^2
 		&+\frac{4c_{1}t}{(2\alpha_{0}+2t+p-2)^2}\int_{\Omega}   \left|\nabla \left(f^{\frac{\alpha_{0}+t-1}{2}+\frac{p}{4} }\eta\right)\right|^2 \\
 		&\leq  2(n-1)K  \int_{\Omega} f^{\alpha_{0}+t+\frac{p}{2}-1}\eta^2
 		+\frac{4(p+1)^2 }{c_{1}t} \int_{\Omega} f^{\alpha_{0}+\frac{p}{2}+t-1}|\nabla\eta|^2\\
 		&+\frac{8c_{1}t}{(2\alpha_{0}+2t+p-2)^2}\int_{\Omega} f^{\alpha_{0}+t+\frac{p}{2}-1}|\nabla\eta|^2 .
 	\end{split}
 \end{align}
 To avoid cumbersome expressions, we need to simplify some coefficients in \eqref{equ 4.13}. Choose $c_{3},c_{4}$  to guarantee
 \begin{align*}
 	\frac{c_{3}}{t}\leq  \frac{4c_{1}t}{(2\alpha_{0}+2t+p-2)^2}\quad\text{and}\quad
 	\frac{8c_{1}t}{(2\alpha_{0}+2t+p-2)^2}+\frac{4(p+1)^2}{c_{1}t} \leq\frac{c_{4}}{t}.
 \end{align*}
 Consequently, it yields that
 \begin{align}
 	\label{equ 4.14}
 	\begin{split}
 		& \sigma \int_{\Omega} f^{\alpha_{0}+\frac{p}{2}+t}\eta^2
 		+
 		\frac{c_3}{t}\int_{\Omega}   \left|\nabla \left(f^{\frac{\alpha_{0}+t-1}{2}+\frac{p}{4} }\eta\right)\right|^2 \\
 		&\leq \, 
 		2(n-1)K  \int_{\Omega} f^{\alpha_{0}+t+\frac{p}{2}-1}\eta^2
 		+
 		\frac{c_{4} }{t} \int_{\Omega} f^{\alpha_{0}+\frac{p}{2}+t-1}\left|\nabla\eta\right|^2.
 	\end{split}
 \end{align}
 From Saloff-Coste's Sobolev embedding inequality, we have
 \[\left\|f^{\frac{\alpha_{0}+t-1}{2}+\frac{p}{4} }\eta\right\|_{L^{\frac{2n}{n-2}}(\Omega)}^2\leq e^{C_n(1+\sqrt{K}R)}V^{-\frac{2}{n}}R^{2}\left(\int_{\Omega}\left|\nabla \left(f^{\frac{\alpha_{0}+t-1}{2}+\frac{p}{4} }\eta\right)\right|^2+R^{-2}\int_{\Omega} f^{ \alpha_{0}+t +\frac{p}{2} -1 }\eta ^2\right). \]
 Inserting the above formula into \eqref{equ 4.14}, we immediately gain 
 \begin{align}\label{equ 4.15}
 	\begin{split}
 		& \sigma \int_{\Omega} f^{\alpha_{0}+\frac{p}{2}+t}\eta^2
 		+
 		\frac{c_{3}}{t }e^{-C_n(1+\sqrt{K}R)}V^{\frac{2}{n}}R^{-2}\left\|f^{\frac{\alpha_{0}+t-1}{2}+\frac{p}{4} }\eta\right\|_{L^{\frac{2n}{n-2}}(\Omega)}^{2}\\
 		&\leq \, 
 		2(n-1)K  \int_{\Omega} f^{\alpha_{0}+t+\frac{p}{2}-1}\eta^2+\frac{c_{4}}{ t }\int_{\Omega} f^{\alpha_{0}+t+\frac{p}{2}-1}|\nabla\eta|^2
 		+\frac{c_{3}}{ t }\int_{\Omega} R^{-2}f^{ \alpha_{0} +\frac{p}{2}+t-1 }\eta ^2.
 	\end{split}
 \end{align}
 Take $t_{0} = c_{p,q,n,\gamma}(1+\sqrt{K} R)$, where $$ c_{p,q,n,\gamma}=\max\left\{C_n+1, \frac{ a_{3}^2}{c_{1}\,\sigma}\right\}.$$
 Now we choose some $t\geq t_{0}$. Notice that
 \begin{align*}
 	2(n-1)K  R^2\leq\frac{ 2(n-1)}{c_{p,q,n,\gamma}^2}t_{0}^2\quad \text{ and}\quad \frac{c_{3}}{t}\leq \frac{c_{3}}{c_{p,q,n,\gamma}},
 \end{align*}
 so we can select a constant $c_5 = c_5(n,p,q,\gamma)>0$ such that
 \begin{align}\label{a5}
 	2(n-1)K  R^2+\frac{c_{3}}{t}\leq c_{5}t_{0}^2 \triangleq c_{5} c_{p,q,n,\gamma}^{2} \left(1+\sqrt{K} R\right)^2.
 \end{align}
 It follows from \eqref{equ 4.15} and (\ref{a5}) that
 \begin{align}\label{integral inequality}
 	\begin{split}
 		&\sigma \int_{\Omega} f^{\alpha_{0}+\frac{p}{2}+t}\eta^2
 		+
 		\frac{c_3}{ t }e^{-t_{0}}V^{\frac{2}{n}}R^{-2}\left\|f^{\frac{\alpha_{0}+t-1}{2}+\frac{p}{4} }\eta\right\|_{L^{\frac{2n}{n-2}}(\Omega)}^2\\
 		\leq & \,
 		c_{5}t_{0}^2R^{-2} \int_{\Omega} f^{\alpha_{0}+\frac{p}{2}+t-1}\eta^2+\frac{c_{4}}{t }\int_{\Omega} f^{\alpha_{0}+\frac{p}{2}+t-1}|\nabla\eta|^2.
 	\end{split}
\end{align}
 So far, we have established the desired integral inequality.
 \subsection{$L^{\gamma}$  estimate of gradient and Moser iteration}\label{section4.2}
 Now we  prove the $L^{\gamma}$ bound of $f$ in a suitable ball and perform the iteration procedure. For the readers' convenience, we  state the whole proof.
 However, some details may be omitted for simplicity. We state such a lemma at the beginning of this section.
 
 \begin{lem}\label{L Beta Bound}
 	Suppose $\omega$ is a positive solution of equation \eqref{equ1 for omega}, \eqref{equ3 for omega} or \eqref{equ2 for omega} on the geodesic ball $B(o,R)\subset M$. Set\;$f=|\nabla \omega|^2,\;\gamma = \left(\alpha_{0}+t_{0}+\frac{p}{2}-1\right)\frac{n}{n-2}$. Let V  be the volume of geodesic ball $B(o,R)$. Then there exists some constant $c_8 = c_8(n,p, q,\gamma)>0$ such that
 	\begin{align}\label{lpbpund}
 		\|f \|_{L^{\gamma}(B(o,3R/4))}\leq c_8V^{\frac{1}{\gamma}} \frac{t_{0}^2}{ R^2}.
 	\end{align}
 \end{lem}
 \begin{proof}
 	Through a careful observation to \eqref{integral inequality}, we divide the region $\Omega$ into two disjoint parts $\Omega_{1}$ and $\Omega_{2}$ as follows
 	\begin{align}\nonumber
 		\Omega=\Omega_{1}\cup \Omega_{2},\quad \Omega_{1}\cap\Omega_{2}=  \varnothing,\quad\Omega_{1}=\left\{f\geq  \frac{2c_{5}t_{0}^2}{\sigma R^{2}}\right\}.
 	\end{align}
 	In $\Omega_{1}$, we have
 	\begin{equation*}
 		c_{5}t_{0}^{2}R^{-2} \int_{\Omega} f^{\alpha_{0}+\frac{p}{2}+t-1}\eta^2
 		\leq
 		\frac{\sigma}{2}\int_{\Omega} f^{\alpha_{0}+\frac{p}{2}+t}\eta^2.
 	\end{equation*}
 	By decomposing $\Omega$, a direct computation yields
 	\begin{align}\label{equ 4.19}
 		\begin{split}
 			c_{5}t_{0}^2R^{-2} \int_{\Omega} f^{\alpha_{0}+\frac{p}{2}+t-1}\eta^2	
 			\leq
 			\frac{\sigma}{2}\int_{\Omega} f^{\alpha_{0}+\frac{p}{2}+t}\eta^2+
 			\frac{2c_{5}t_{0}^2}{R^2} \left(\frac{2c_{5}t_{0}^2}{\sigma R^2}\right)^{\alpha_{0}+\frac{p}{2} +t-1 }V.
 		\end{split}
 	\end{align}
 	Combining \eqref{integral inequality} and \eqref{equ 4.19}, and choosing $t=t_{0}$, we get	  
 	\begin{align}\label{equ 4.20}
 		\begin{split}
 			&\frac{\sigma}{2}\int_{\Omega} f^{\alpha_{0}+\frac{p}{2}+t_{0}}\eta^2+ \frac{c_3}{ t_{0} }e^{-t_{0}}V^{\frac{2}{n}}R^{-2}\left\|f^{\frac{\alpha_{0}+t_{0}-1}{2}+\frac{p}{4} }\eta\right\|_{L^{\frac{2n}{n-2}}(\Omega)}^2\\
 			&\leq \frac{c_{5}t_{0}^2}{R^2} \left(\frac{2c_{5}t_{0}^2}{ \sigma R^2}\right)^{\alpha_{0}+\frac{p}{2} +t_{0}-1 }V
 			+\frac{c_{4}}{ t_{0} }\int_{\Omega} f^{\alpha_{0}+\frac{p}{2}+t_{0}-1}|\nabla\eta|^2.
 		\end{split}
 	\end{align}
 	Now we select the function $\eta=\eta_{0}^{\alpha_{0}+\frac{p}{2}+t_{0}}$, where $\eta_{0}\in C^{\infty}_{0}(B(o,R))$ satisfies 
 	\[\begin{cases}
 		0\leq\eta_0\leq 1,\quad \eta_0\equiv 1\text{ in }  B(o,\frac{3R}{4});\\
 		|\nabla\eta_0|\leq\frac{\tilde{C}(n)}{R}.
 	\end{cases}\]
 	It can be easily seen that
 	\begin{align*}
 		c_{4}R^{2} |\nabla\eta|^2\leq c_{4}\tilde{C}^{2} \left(\alpha_{0}+\frac{p}{2}+t_{0}\right )^2\eta ^{\frac{2\alpha_{0}+2t_{0}+p-2}{\alpha_{0}+p/2+t_{0}}}
 		\leq c_{6}t^2_0\eta^{\frac{2\alpha_{0}+p+2t_{0}-2}{\alpha_{0}+p/2+t_{0}}}.
 	\end{align*}
 	By using H\"{o}lder's inequality and Young's inequality, we get 
 	\begin{align}\label{equ 4.21}
 		\begin{split}
 			\frac{c_{4}}{t_{0}}\int_{\Omega}f^{\frac{p}{2}+\alpha_{0}+t_{0}-1}|\nabla\eta|^2
 			\leq &\frac{c_6t_{0}}{R^2}  \left(\int_{\Omega}f^{\alpha_{0}+t_{0}+ \frac{p}{2} }\eta^2\right)^{\frac{\alpha_{0}+p/2+t_{0}-1}{\alpha_{0}+p/2+t_{0}}}V^{\frac{ 1}{\alpha_{0}+t_{0}+ p/2 }}\\
 			\leq &\frac{\sigma}{2}\left(\int_{\Omega}f^{ \alpha_{0}+t_{0}+\frac{p}{2} }\eta^2 + \left(\frac{2c_{6}t_{0} }{\sigma R^2}\right)^{ \alpha_{0}+t_{0}+p/2 }V\right).
 		\end{split}
 	\end{align}
 	Substituting \eqref{equ 4.20} and \eqref{equ 4.21} into \eqref{integral inequality}, we obtain
 	\begin{align}\label{equ 4.22}
 		\begin{split}
 			&\left(\int_{\Omega}f^{\frac{n(\alpha_{0}+p/2+t_{0}-1)}{n-2}}\eta^{\frac{2n}{n-2}}\right)^{\frac{n-2}{n}}\\
 			\leq &\frac{t_{0}}{c_3} e^{t_{0}}V^{1-\frac{2}{n}}R^2\left[\frac{2c_5t_{0}^2}{R^2} \left(\frac{2c_{5}t_{0}^2}{\sigma R^2}\right)^{\alpha_{0}+t_{0}+\frac{p}{2} -1 } + \frac{c_{6}t_{0}}{R^2} \left(\frac{2c_{6}t_{0} }{\sigma R^2}\right)^{\alpha_{0}+t_{0}+\frac{p}{2} -1 }\right]\\
 			\leq &c_7^{t_{0}}e^{t_{0}}V^{1-\frac{2}{n}}t_{0}^3\left( \frac{t_{0}^2}{ R^2}\right)^{\alpha_{0}+t_{0}+\frac{p}{2} -1}.
 		\end{split}
 	\end{align}
 	We pick $c_{7}$ to meet the condition
 	\[c_{7}^{t_{0}} \geq \frac{2c_{5}}{c_{3}}\left(\frac{2c_{5}}{\sigma }\right)^{\alpha_{0}+t_{0}+\frac{p}{2}-1}
 	+\frac{c_{6}}{c_3}\left(\frac{2c_{6}}{\sigma t_{0}}\right)^{\alpha_{0}+t_{0}+\frac{p}{2}-1}.\]
 	Let
 	$$
 	c_8 = t_0^{\frac{3}{\alpha_{0}+t_0+p/2-1}}c_{7}^{\frac{t_0}{\alpha_{0}+t_0+p/2 -1}}
 	$$
 	and take the $1/(\alpha_{0}+\frac{p}{2}+t_{0}-1)$ root on both sides of \eqref{equ 4.22}. After a proper simplification, there holds
 	\begin{align}
 		\left\|f\eta^{\frac{2}{\alpha_{0}+t_0+p/2-1}}\right\|_{L^{\gamma}(\Omega)}\leq c_8V^{\frac{1}{\gamma}} \frac{t_0^2}{ R^2}.
 	\end{align}
 	The above formula implies \eqref{lpbpund} notably. 	
 \end{proof} 	 
 Then we execute so-called Nash-Moser iteration, which suggests  $L^{\infty}$ bound of $f$ in the ball $B(o,\frac{R}{2})$.
 \begin{lem}\label{L infinity bound}
 	Let $(M,g)$ be a complete Riemannian  manifold satisfying $Ric(M)\ge -(n-1)K g$ for some $K\ge 0$. Denote $f=|\nabla \omega|^2$. Under the same assumptions as in $\lemref{L Beta Bound}$, there exists $c_{11} = c_{11}(n,p,q,\gamma)>0$ such that
 	\begin{align*}
 		\|f\|_{L^{\infty}(B(o,R/2))}\leq & c_{11}\frac{(1+\sqrt{K}R)^2}{R^2}.
 	\end{align*}
 \end{lem}
 \begin{proof}
 	By neglecting the term involving $\sigma$ in \eqref{integral inequality}, we obtain
 	\begin{align}\label{equ 4.24}
 		\begin{split}
 			&\frac{c_3}{t}e^{-t_{0}}V^{\frac{2}{n}}R^{-2}\left\|f^{\frac{\alpha_{0}+t-1}{2}+\frac{p}{4} }\eta\right\|_{L^{\frac{2n}{n-2}}(\Omega)}^2
 			\\&	\leq 
 			c_{5}t_{0}^2R^{-2} \int_{\Omega} f^{\alpha_{0}+\frac{p}{2}+t-1}\eta^2+\frac{c_4}{t }\int_{\Omega} f^{\alpha_{0}+\frac{p}{2}+t-1}|\nabla\eta|^2.		
 		\end{split}
 	\end{align}
 	Some of the necessary settings are described below. Set
 	\begin{equation*}
 		\Omega_{k}=B(o,r_{k}), \quad \,\text{where} \;\, r_{k}=\frac{R}{2}+\frac{R}{4^{k}}\;\;\text{and}\;\,k\in \BN_{+}\,.
 	\end{equation*}
 	We choose $\eta_{k}$ such that
 	\begin{equation*}
 		\eta_{k}\in C^{\infty}(\Omega_{k}), \; 0\le\eta_{k}\le 1,\; \eta_{k}\equiv\,1\;\text{in}\;\Omega_{k+1}\;\,\text{and}\;|\nabla \eta_{k}|\le \frac{\tilde{C}(n)4^{k}}{R}.
 	\end{equation*}
 	We replace $\eta$ in \eqref{equ 4.24} by $\eta_{k}$ to gain
 	\begin{equation*}
 		c_{3}e^{-t_{0}}V^{\frac{2}{n}} \left\|f^{\frac{\alpha_{0}+t-1}{2}+\frac{p}{4} }\eta_{k}\right\|_{L^{\frac{2n}{n-2}}(\Omega_{k})}^2
 		\leq  \left(c_{5}t_{0}^{2}t +c_{4} \tilde{C}^{2}16^{k}\right)\int_{\Omega_{k}} f^{\alpha_{0}+\frac{p}{2}+t-1}.
 	\end{equation*}
 	To give the iterative formula, we set $\gamma_{1}=\gamma,\gamma_{k+1}=\frac{n\gamma_{k}}{n-2}$ and let $t=t_{k}$ to ensure \[ t_{k}+\frac{p}{2}+\alpha_{0}-1=\gamma_{k}.\]
 	Then we deduce that 
 	\begin{align*}
 		c_{3} \left(\int_{\Omega_{k}}f^{\gamma_{k+1}}\eta_k^\frac{2n}{n-2}\right)^{\frac{n-2}{n}}
 		\leq & e^{ t_{0}}V^{-\frac{2}{n}}\left(c_{5}t_{0}^2\left(t_{0}+\frac{p}{2}+\alpha_{0}-1\right)\left(\frac{n}{n-2}\right)^k + c_{4} \tilde{C}^{2}16^{k}\right)\int_{\Omega_{k}} f^{\gamma_{k}},
 	\end{align*}
 	Meanwhile, we choose $c_{9}$ satisfying
 	$$
 	c_{3}c_{9}t_{0}^3\geq \max\left\{c_{5}t_{0}^2\left(\alpha_{0}+t_0+\frac{p}{2}-1\right),\, c_{4} \tilde{C}^{2} \right\}.
 	$$
 	Since $\frac{n}{n-2}<16$, we deduce that
 	\begin{align}\label{equ 4.25}
 		\left(\int_{\Omega_{k}}f^{\gamma_{k+1}}\eta_k^\frac{2n}{n-2}\right)^{\frac{n-2}{n}}
 		\leq &2c_{9}t_{0}^3 e^{ t_{0}}V^{-\frac{2}{n}} 16^{k} \int_{\Omega_{k}} f^{\gamma_{k}}.
 	\end{align}
 	Taking both sides of \eqref{equ 4.25} by the power $\frac{1}{\gamma_{k}}$, it yields
 	\begin{align}\label{iteration formula}
 		\|f\|_{L^{\gamma_{k+1}}(\Omega_{k+1})}
 		\leq &\left(2c_{9}t_{0}^{3} e^{ t_{0}}V^{-\frac{2}{n}}\right)^{\frac{1}{\gamma_{k}}} 16 ^{\frac{k}{\gamma_{k}}}\|f\|_{L^{\gamma_{k}}(\Omega_{k})}.
 	\end{align}
 	We now perform standard Moser iteration using \eqref{iteration formula}. Since both series
 	\[ \sum_{k=1}^{\infty}\frac{1}{\gamma_{k}} \quad and \quad \sum_{k=1}^{\infty}\frac{k}{\gamma_{k}} \]
 	are convergent, it follows that
 	\begin{align}\label{iteration result}
 		\begin{split}
 			\|f\|_{L^{\infty}(B(o,R/2))}\leq &\left(2 c_{9}t_{0}^3 e^{ t_{0}} \right)^{\sum_{k=1}^{\infty}\frac{1}{\gamma_{k}}} 16 ^{\sum_{k=1}^{\infty}\frac{k}{\gamma_{k}}}  V^{-\frac{1}{\gamma}} \|f\|_{L^{\gamma}(B(o,3R/4))} \\ \triangleq& \;c_{10}V^{-\frac{1}{\gamma}} \|f\|_{L^{\gamma}(B(o,3R/4))}.
 		\end{split}			
 	\end{align}
 	Finally, we substitute \eqref{lpbpund} into \eqref{iteration result} and get
 	\begin{equation*}
 		\|f\|_{L^{\infty}(B(o,R/2))}\leq c_{8}c_{10}c_{0}^{2} \frac{(1+\sqrt{K}R)^2}{R^2}\triangleq c_{11}\frac{(1+\sqrt{K}R)^2}{R^2}. \qedhere
 	\end{equation*}
 \end{proof}
 		\section{Proof of the main theorem}	
 	\noindent{\bf Proof of \thmref{thm1.1}}
 	In \secref{standard section}, we have concluded that 
 	\begin{equation}\label{estiamte for gradient omega}
 		\sup_{B(o,\frac{R}{2})}|\nabla \omega|\le c_{12}\frac{1+\sqrt{K}R}{R},
 	\end{equation}
 	where the constant $c_{12}$ depends on $n,p,q,\gamma$.
 	Recall when $\gamma\neq 1/(p-1)$,
 	\[v=\frac{\gamma}{b}u^{b},\qquad \omega=-(p-1)\log v.\]
 	
 	Since $u$ is a positive solution to \eqref{main equ}, by combining Lemmas \ref{lem3.1}, \ref{lem3.2}, \ref{L Beta Bound} and \ref{L infinity bound}, we transfer $\omega$ into $u$ to obtain
 	$$\sup_{B(o,\frac{R}{2})} \frac{|\nabla u|}{u}\leq C\frac{(1+\sqrt K R)}{R}, $$
 	where $C=C(n,p,q,\gamma)$. This accomplishes the proof.\qed 
 	\vspace{0.5em}\\
 	\noindent{\bf Proof of \thmref{thm1.2}}\,
 	Actually, this proof is just a modification of the proof of \thmref{thm1.1}. Instead of \lemref{lem3.1}, we use
 	\lemref{lem3.5} here. \qed 

\vspace{0.5em}
\noindent{\bf Proof of \thmref{thm1.3}} 
 	Remember when $\gamma= 1/(p-1)$,
 	\[v=\frac{1}{p-1}\log u,\qquad \omega= v.\]
 	
 	A direct combination of Lemmas \ref{lem3.3}, \ref{lem3.4},
 	\ref{L Beta Bound} and \ref{L infinity bound} proves the conclusion. The remaining details follow the proof of \thmref{thm1.1} and are therefore omitted. \qed 
 	\vspace{0.5em}\\
 	\noindent{\bf Proof of \corref{cor1.4}} We always assume
 	$a\neq 0$ throughout this proof. A direct combination of \eqref{condition2 in thm1.3} and \eqref{parameter restricition3} yields it in the case $\gamma=1/(p-1)$. For $\gamma\neq 1/(p-1)$, by joining \eqref{parameter restriction1} and \eqref{parameter restriction2}, we have
 	\begin{equation}\label{equ 5.2}
 		a>0,\quad q\le \frac{n+1}{n-1}\gamma(p-1),
 	\end{equation}
 	and
 	\begin{equation}\label{equ 5.3}
 		a<0,\quad q\ge \frac{n+1}{n-1}\gamma(p-1).
 	\end{equation}
 	Since \eqref{condition2 in thm1.1} and \eqref{condition2 in thm1.2} are identical, combining \eqref{equ 5.2} and \eqref{condition2 in thm1.1} leads to 
 		\[ a>0 \quad \text{and} \quad q<\frac{n+3}{n-1}\gamma(p-1).\]
 		
 	Similarly, the union of \eqref{equ 5.3} and \eqref{condition2 in thm1.1} is
 	$a<0 \; \text{and}\; q>\gamma(p-1).$ This ends the whole proof. \qed 
 	\vspace{0.5em}\\
 	\noindent{\bf Proof of \corref{liouville}} We take $K=0$ in one of the above theorems and immediately obtain 
 	\begin{equation}\label{equ 5.4}
 	\sup_{B(o,R/2)} \frac{|\nabla u|}{u}\le \frac{C(n,p,q,\gamma)}{R}.
 	\end{equation}
 	This implies $|\nabla  u|=0$ if $R\to \infty$ in \eqref{equ 5.4}. Then $u$ is a constant and $\Delta_{p}(u^{\gamma})=0$. However, this contradicts to \eqref{main equ} since $u$ is positive. \qed 
 	\vspace{0.5em}\\
 	\noindent{\bf Proof of \corref{harnack}}\,
 	From the above theorems, for any $y\in B(o,R/2)$, we have
 	\begin{align}\label{equ 5.5}
 		|\nabla \log u(y)| 
 		\leq \frac{ C(n,p,q,\gamma)(1+\sqrt{K}R) }{R }.
 	\end{align}
 	Choose  a minimizing geodesic $\gamma(t)$ with arc length parameter connecting $o$ and $y$, i.e.
 		$$
 	\gamma:[0,d]\to M,\quad\gamma(0)=o, \quad \gamma(d)=x.
 	$$
 	Notice that $d=d(x,o)\le \frac{R}{2}$ is the geodesic distance, we know
 	\begin{align}\label{equ 5.6}
 		\log u(y)-\log u(o)=\int_0^d\frac{d}{dt}\log  u\circ\gamma(t)\,\mathrm{d}t.
 	\end{align}
 	Since
 	\begin{align}\label{equ 5.7}
 		\left|\frac{d}{dt}\log u\circ\gamma(t)\right|\leq |\nabla \log u||\gamma'(t)| \leq \frac{ C(n,p,q,\gamma)(1+\sqrt{K}R) }{R },
 	\end{align}
 	it infers from \eqref{equ 5.6} and \eqref{equ 5.7} that
 	\begin{equation*}
 		-C(n,p,q,\gamma)\frac{1+\sqrt{K}R}{2}\le \log \frac{u(x)}{u(o)}\le C(n,p,q,\gamma)\frac{1+\sqrt{K}R}{2}.
 	\end{equation*}
 	As a consequence, for any $y,z\in B(o,R/2)$, we have
 	\begin{equation*}
 		\log\frac{u(z)}{u(y)}\le C(n,p,q,\gamma)(1+\sqrt{K}R).
 	\end{equation*}
 	If we consider global solution $u$ on $M$, we firstly let $R\to \infty$ in \eqref{equ 5.2} and obtain that
 	$$
 	|\nabla \log u(y)|\leq C \sqrt{K}, \quad \forall\, y\in M.
 	$$
 	Fix  $y\in M$, it is known  that for any $z\in M$, we can choose a geodesic  $\gamma=\gamma(t)$ which minimizes the line between $y$ and $z$
 	$$
 	\gamma:[0,d]\to M,\quad\gamma(0)=y, \quad \gamma(d)=z,
 	$$
 	where $d=d(y,z)$ denotes the distance from $y$ to $z$.
 	There holds true
 	\begin{align}\label{equ 5.8}
 		\log u(z)-\log u(y)=\int_0^d\frac{d}{dt}\log u\circ\gamma(t)\,\mathrm{d}t.
 	\end{align}
 	Due to
 	\begin{align}\label{equ 5.9}
 		\left|\frac{d}{dt}\log u\circ\gamma(t)\right|\leq |\nabla \log u||\gamma'(t)| = C\sqrt{K},
 	\end{align}
 	it follows from \eqref{equ 5.8} and \eqref{equ 5.9} that
 	\begin{equation*}
 		\frac{u(z)}{u(y)}\le e^{C(n,p,q,\gamma)\sqrt{K}d(y,z)}.
 	\end{equation*}
    Thus we complete the proof. \qed \\
 
 \noindent{\bf The case $n=2$.} In the proof of above theorems,  we  used Sobolev embedding inequality \eqref{salof}, which requires the dimension $n>2$. As a necessary supplement, let us examine the case $n=2$. Instead of $\eqref{salof}$, we will use the following inequality for any fixed $\hat{n}>2$ and $f\in C^{\infty}_{0}(B)$
 \begin{equation}\label{two dim inequ}
 	 \|f\|_{L^{\frac{2\hat{n}}{\hat{n}-2}}}^2\leq e^{C_{\hat{n}}(1+\sqrt{K}R)}\,V^{-\frac{2}{\hat{n}}}R^2\left(\int|\nabla f|^2+R^{-2}f^2\right).
 \end{equation}
 We choose $\hat{n}=3$ without loss of generality. Then \eqref{two dim inequ} becomes
 \begin{equation}\nonumber
 	 \|f^{\frac{\alpha_{0}+t-1}{2}+\frac{p}{4} }\eta \|_{L^{6}(B)}^2\leq e^{C(1+\sqrt{K}R)}\,V^{-\frac{2}{\hat{3}}}R^2\left(\int\left|\nabla \left(f^{\frac{\alpha_{0}+t-1}{2}+\frac{p}{4} }\eta\right) \right|^2+R^{-2}\int f^{\alpha_{0}+t+\frac{p}{2}-1 }\eta ^2\right),
 \end{equation}
 where $f$ is replaced by $f^{\frac{\alpha_{0}+t-1}{2}+\frac{p}{4} }\eta$ as in \eqref{equ 4.12}.
 
 \noindent Proceeding in a similar manner, we infer from \eqref{equ 4.14} and the above formula  that
 \begin{align}\label{another1}
 	\begin{split}
 		& \sigma \int_{\Omega} f^{\alpha_{0}+\frac{p}{2}+t}\eta^2
 		+
 		\frac{c_{3}}{t }e^{-C(1+\sqrt{K}R)}V^{\frac{2}{3}}R^{-2}\left\|f^{\frac{\alpha_{0}+t-1}{2}+\frac{p}{4} }\eta\right\|_{L^{6}(\Omega)}^{2}\\
 		&\leq \, 
 		\left(2(n-1)K+\frac{c_{3}}{tR^{2}} \right) \int_{\Omega} f^{\alpha_{0}+t+\frac{p}{2}-1}\eta^2+\frac{c_{4}}{ t }\int_{\Omega} f^{\alpha_{0}+t+\frac{p}{2}-1}|\nabla\eta|^2.
 	\end{split}
 \end{align}
A similar treatment leads to the following integral inequality
 \begin{align}\label{another2}
	\begin{split}
		&\sigma \int_{\Omega} f^{\alpha_{0}+\frac{p}{2}+t}\eta^2
		+
		\frac{c_3}{ t }e^{-t_{0}}V^{\frac{2}{3}}R^{-2}\left\|f^{\frac{\alpha_{0}+t-1}{2}+\frac{p}{4} }\eta\right\|_{L^{6}(\Omega)}^{2}\\
		\leq & \,
		c_{5}t_{0}^2R^{-2} \int_{\Omega} f^{\alpha_{0}+\frac{p}{2}+t-1}\eta^2+\frac{c_{4}}{t }\int_{\Omega} f^{\alpha_{0}+\frac{p}{2}+t-1}|\nabla\eta|^2.
	\end{split}
\end{align}
By repeating the previous procedure in \secref{section4.2}, we easily obtain the corresponding $L^{\gamma}$ bound estimate of $f$, i.e.
\begin{equation}\label{another3}
		\|f \|_{L^{\gamma}(B(o,3R/4))}\leq c_8V^{\frac{1}{\gamma}} \frac{t_{0}^2}{ R^2},
\end{equation}
where $\gamma=3\left(\alpha_{0}+t_{0}+p/2-1\right)$.

We take $\gamma_{1}=\gamma,\gamma_{k+1}=3\gamma_{k}$ and $\Omega_{k}$ defined as above when iteration is performed.
Carrying out Nash-Moser iteration, we obtain
\begin{equation}\label{another4}
		\|f\|_{L^{\infty}(B(o,R/2))}\leq c_{10}V^{-\frac{1}{\gamma}} \|f\|_{L^{\gamma}(B(o,3R/4))}\,.
\end{equation}

The gradient estimate is a direct combination of $\eqref{another3}$ and \eqref{another4}. We can also prove Harnack's inequality and Liouville type results similarly and the details will be omitted here.
\section{Spceial Case: $a=0$ }
	This section is further contributed to investigating the homogeneous equation
	\begin{equation}\label{special main equ}
		\Delta_{p}(u^{\gamma})=0.
	\end{equation}
The local gradient estimate for $v$, defined by \eqref{equ 2.1}, is addressed in the following analysis.
	Making nonlinear ``pressure'' transformation as well, equation \eqref{special main equ} is turned into
	\begin{align}
		&\Delta_{p}v+b^{-1}v^{-1}|\nabla v|^{p}=0, &b\neq 0;\label{equation1 for special v} \\
		&\Delta_{p}v+(p-1)|\nabla v|^{p}=0,& b=0.\label{equation2 for special v} 
	\end{align}
	
	 Despite having obtained stronger conclusions through Theorems \ref{thm1.1}, \ref{thm1.2} and \ref{thm1.3}, we deliberately focus on this special case because equations \eqref{equation1 for special v} and \eqref{equation2 for special v} share identical structural properties --\, a critical alignment formalized in \thmref{thmB}.  Let us consider what we can gain from this theorem. For $b\neq 0$, setting $\beta=b^{-1},q=p$ and $r=-1$ in \thmref{thmB}, a straightforward calculation yields
	\begin{gather}
			\beta\left(\frac{n+1}{n-1}-\frac{q+r}{p-1}\right)=\frac{2}{b(n-1)},\\
	\intertext{and}
			\frac{q+r}{p-1}\equiv 1, \quad\forall \beta\in \BR.
	\end{gather}
	For $b=0$, denoting $\beta=(p-1), r=0$ and $q=p$, a similar calculation shows that
	\begin{gather}
		\beta\left(\frac{n+1}{n-1}-\frac{q+r}{p-1}\right)=(p-1)\left(\frac{n+1}{n-1}-\frac{p}{p-1}\right),\\
		\intertext{and}
		\frac{q+r}{p-1}=\frac{p}{p-1}, \quad\forall \beta\in \BR.
	\end{gather}
	A direct application of \thmref{thmB} implies the following result.
	\begin{thm}\label{thm6.1}
		Let $(M^{n},g)$ be a  complete Riemannian manifold with $Ric(M)\ge -(n-1)K g$ for some $K\ge 0$. Assume $u$
		is a positive solution to \eqref{special main equ}. Then we have the following statements.
		\begin{enumerate} [$(I)$.]
		\item  If $\gamma>\frac{1}{p-1}$, then $v$ satisfies local gradient estimate \eqref{CY estimate for p-laplace} with $C$ depending on $n,p,\gamma$ instead;
		\item If  $\gamma=\frac{1}{p-1}$, with additional assumptions $u>1$ and $p>\frac{n+3}{4}>1$, estimate in (I) still holds for $v$.
	\end{enumerate}
	\end{thm}
		\begin{rmk}
		Since $v$ still keeps positive for $\gamma>\frac{1}{p-1}$, henceforth none of assumption on the lower bound of $u$ is required. The equivalence \[\frac{|\nabla v|}{v}=\frac{b|\nabla u|}{u}.\] implies that u inherits a variant of \eqref{CY estimate for p-laplace} under the condition $\gamma>\frac{1}{p-1}$.
	\end{rmk}
	\begin{rmk}
		When $p>2$, the result in (I) generalizes the range $\gamma\ge 1$ that is derived directly by applying \thmref{thmB} to \eqref{special main equ}, while it fails in the case $1<p\le 2$. In addition, compared with known conclusion for $\gamma=1${\rm\cite{WZ}}, our result can only hold for $p>2$, which suggests it is not sharp.
	\end{rmk}
	\begin{rmk}
		By using \thmref{thmB} to \eqref{special main equ}  for $\gamma=1/(p-1)$, $p$ should be restricted in the range $1<p\le 2$. So the result in (II) behaves better for some large $p$.
		However, there is a constraint on the lower bound of the solution $u$.
	\end{rmk}
	
	It is unfortunate  that \thmref{thm6.1} contains no case about $b<0$. Note that $bv>0$ holds now and \eqref{equation1 for special v} 
	can be rewritten as follows
	\begin{equation}\label{new 6.2}
		bv\Delta_{p}v+|\nabla v|^{p}=0.
	\end{equation}
	Surprisingly,  the following Caccioppoli type inequality holds for $b<0$.
	\begin{thm}\label{Caccioppoli}\label{thm6.2}
		If $v$ is a (weak) solution of \eqref{new 6.2} with $b<0$, then
		\begin{equation}\label{equ 6.9}
			\int_{M} |\nabla v|^{p}\eta^{p}\le C(p,b)\int_{M} |v|^{p}|\nabla \eta|^{p}
		\end{equation}
	holds for any $\eta\in C^{\infty}_{0}(M)$, where $C(p,b)=\left(2p|b|/(1-b)\right)^{p}>0$. In particular, on some geodesic ball $B_{2R}\in M$, we have
		\begin{equation}\label{estiamte6.10}
			\int_{B_{R}} |\nabla v|^{p}\le C(n,p,b)R^{-p}\int_{B_{2R}} |v|^{p}.
		\end{equation}
	\end{thm} 
	\begin{proof}
		Multiplying both sides of \eqref{new 6.2} by $\eta^{p}$ and integrating on $M$, we get
		\begin{equation*}
			\int_{M} b\Delta_{p}v\cdot v\eta^{p}+|\nabla v|^{p}\eta^{p}=0.
		\end{equation*}
	    Integrate by parts to show
		\begin{equation}
			(1-b)\int_{M} |\nabla v|^{p}\eta^{p}-bp\int_{M} v\eta^{p-1}|\nabla v|^{p-2}\nabla v\cdot \nabla\eta=0.
		\end{equation}
		
		Due to $b<0$ and $bv>0$, we use H\"{o}lder's inequality with coefficients $p$ and $p/(p-1)$ to get
		\begin{align}\label{equ 6.12}
			\begin{split}
			(1-b)\int_{M} |\nabla v|^{p}\eta^{p}&\le p|b|\int_{M}|v|\eta^{p-1}|\nabla v|^{p-1}|\nabla\eta|\\
			&=p|b|\int_{M} \varepsilon^{\frac{1}{p}}|\nabla v|^{p-1}\eta^{p-1}\cdot \varepsilon^{-\frac{1}{p}}|v||\nabla \eta| \\
			&\le \varepsilon^{\frac{1}{p-1}}\int_{M} |\nabla v|^{p}\eta^{\,p}+\frac{p^{\,p}|b|^{p}}{\varepsilon}\int_{M}|v|^{p}|\nabla\eta|^{p}.
			\end{split}
		\end{align}
		Choose $\varepsilon$ such that $\varepsilon^{\frac{1}{p-1}}=\frac{1-b}{2}$, after some simplification \eqref{equ 6.12} becomes
			\begin{equation*}
			\int_{M} |\nabla v|^{p}\eta^{\,p}\le \left(\frac{2p|b|}{1-b}\right)^{p}\int_{M} |v|^{p}|\nabla \eta|^{p}.
		\end{equation*}
		
		This directly proves \eqref{equ 6.9}. From now on we choose some explicit $\eta\in C^{\infty}_{0}(B_{2R})$ as follows,
		\[\begin{cases}
		0\leq\eta\leq 1,\quad \eta\equiv 1\text{ in }  B_{R},\\
		|\nabla\eta|\leq\frac{C(n)}{R} \;\; \text{in}\; B_{2R}.
		\end{cases}\]
		Substitute $\eta$ into the above estimate, it yields
		\begin{equation*}
			\int_{B_{R}} |\nabla v|^{p}\le C(n,p,b)R^{-p}\int_{B_{2R}\setminus B_{R}} |v|^{p}.
		\end{equation*}
		This gives \eqref{estiamte6.10}.
	\end{proof}
	Moreover, we have such a corollary as a direct consequence of \eqref{estiamte6.10}.
	\begin{cor}\label{cor6.3}
		Same conditions and notations as in \thmref{thm6.2}, if $v\in L^{p}_{loc}(M)$ with $p>n$, where $n$ is the dimension of $M$. Then any solution $u$ of \eqref{special main equ} on $M$ is indeed a constant.
		Similar result also holds if conditions are replaced by $Ric(M)\ge 0$, $p>n$ and $v$ is bounded from above by some  constant $A>0$.
	\end{cor}
	\begin{proof}
		From $v\in L^{p}_{loc}(M)$ and \eqref{estiamte6.10}, we have
		\begin{equation}\label{equ 6.13}
			\int_{B_{R}} |\nabla v|^{p}\le C(n,p,b)R^{n-p}\int_{B_{1}}|v|^{p}.
		\end{equation}
		
		Since the integral in the right side of \eqref{equ 6.13} is finite, by letting $R\to \infty$ we know $|\nabla v|\equiv 0$. Thus $|\nabla u|\equiv 0$. Similarly, from standard volume comparison theorem and $v\le A$, \eqref{estiamte6.10} becomes
			\begin{equation}\label{equ 6.14}
			\int_{B_{R}} |\nabla v|^{p}\le C(n,p,b,A)R^{n-p}.
		\end{equation}
		The the whole proof is complete by letting $R\to \infty$ in \eqref{equ 6.14} as well.
	\end{proof}
	The following theorem complements the case $\gamma<\frac{1}{p-1}$.
\begin{thm}\label{thm6.4}
	Let $(M^{n},g)$ be a  complete Riemannian manifold with $Ric(M)\ge -(n-1)K g$ for some $K\ge 0$. Assume $u$
	is a positive solution to \eqref{special main equ} with $0<\gamma<\frac{1}{p-1}$. If $u<\Lambda$ and $1<p\le \frac{n+1}{2}$ or $p>\frac{n+3}{4}$, where
	$$\Lambda=\left(-\frac{\gamma}{b}\right)^{-\frac{1}{b}},$$
	then the following estimate holds
		$$
	\sup_{B(o,\frac{R}{2})} \frac{|\nabla \log(-v)|}{\log(-v)}\leq C(n,p,\gamma)\frac{(1+\sqrt K R)}{R}.
	$$
\end{thm}
\begin{proof}
	Notice that $v<0$ in \eqref{equation1 for special v},  we set $\tilde{v}=-v$ and $\tilde{v}$ is a positive solution of
	\begin{equation*}
		\Delta_{p}\tilde{v}+b\tilde{v}^{-1}|\nabla \tilde{v}|^{p}=0.
	\end{equation*}
	Then let $\omega=\frac{1}{p-1}\log \tilde{v}$ and $\omega$ satisfies
	\begin{equation}\label{equ 6.15}
		\Delta_{p}\omega+\left(1+\frac{1}{b(p-1)}\right)|\nabla \omega|^{p}=0.
	\end{equation}
	Since $0<\gamma<\frac{1}{p-1}$ implies that
	\begin{equation*}
		1+\frac{1}{b(p-1)}<0
	\end{equation*}
	and $u<\Lambda$ suggests $v<-1$, consequently $\omega$ is a positive solution of \eqref{equ 6.15}. In order to apply \thmref{thmB} again, we need to ensure 
	\begin{gather}
		\frac{n+1}{n-1}-\frac{p}{p-1}\le 0 \label{equ 6.16}
	\intertext{or}
	\frac{p}{p-1}<\frac{n+3}{n-1}.   \label{equ 6.17}
	\end{gather}
	Because \eqref{equ 6.16} implies $1<p\le \frac{p+1}{2}$ and \eqref{equ 6.17} equals to $p>\frac{n+3}{4}$. From \thmref{thmB}, $\omega$ satisfies the local gradient estimate
			$$
	\sup_{B(o,\frac{R}{2})} \frac{|\nabla \omega|}{\omega}\leq C(n,p,\gamma)\frac{(1+\sqrt K R)}{R}.
	$$
	Back to $v$, this ends the proof.
\end{proof}

\section*{Acknowledgments}{ This work was partially supported by the National Natural Science Foundation of China (No. 12271423) and the Shaanxi Fundamental Science Research Project for Mathematics and Physics (No. 23JSY026).}

	\bibliographystyle{amsplain}
	\bibliography{reference}
\end{CJK}
\end{document}